\documentclass[11pt]{amsart}
\usepackage[utf8]{inputenc}
\usepackage[T1]{fontenc}
\usepackage{lmodern}
\usepackage{amsmath,amssymb,amsthm}
\usepackage[shortlabels]{enumitem}
\usepackage{xcolor}
\usepackage[colorlinks,linkcolor={red!50!black},citecolor={green!40!black}]{hyperref}
\usepackage[capitalise,nameinlink]{cleveref}
\usepackage[margin=30mm]{geometry}
\usepackage{tikz}
\usetikzlibrary{calc,shapes.geometric,bbox}
\newtheorem{theorem}{Theorem}[section]
\newtheorem{lemma}[theorem]{Lemma}

\theoremstyle{definition}
\newtheorem{definition}[theorem]{Definition}

\newtheorem{problem}[theorem]{Problem}
\crefname{problem}{Problem}{Problems}

\newcommand{\eps}{\varepsilon}
\newcommand{\N}{\mathbb{N}}

\newcommand{\upcl}[1]{\lfloor #1\rfloor}
\newcommand{\downcl}[1]{\lceil #1\rceil}
\newcommand{\sdowncl}[1]{\mathring{\lceil #1\rceil}}
\newcommand{\Tc}{\mathbb{T}}
\newcommand{\Tb}{\mathbb{T}^\flat}

\newcommand{\BG}{\mathbf{G}}
\newcommand{\G}{\mathbf{G}^\flat}
\newcommand{\RG}{\mathbf{R}}
\newcommand{\conc}{{}^\frown}
\DeclareMathOperator{\im}{im}
\DeclareMathOperator{\last}{last}
\DeclareMathOperator{\dom}{dom}

\title[Thin ends and topological spanning trees]{Thin ends and topological spanning trees}
\author{Nathan Bowler}
\author{Max Pitz}
\begin{document}

\begin{abstract}
Halin conjectured in 1964 that every connected graph has an end-faithful spanning tree.
He proved this for countable graphs, but in the early 1990s Seymour and Thomas and, independently, Thomassen found uncountable counterexamples.
Both are infinitely connected and hence one-ended, and their unique end has uncountable degree (i.e.\ contains uncountably many pairwise disjoint rays).

Several authors have since asked how much of the conjecture survives.
Diestel asked in 1992 whether a graph has an end-faithful spanning tree as soon as all its one-ended connected subgraphs have one; by a theorem of Polat, this would give end-faithful spanning trees in all connected graphs whose ends have countable degree.
Diestel and K\"uhn introduced topological spanning trees, a topological variant of end-faithful spanning trees that exists in every one-ended connected graph, and asked in 2004 whether every connected graph has one.
And Carmesin asked in 2019 whether every graph has a spanning tree that is end-faithful at least for its thin ends, the ends of finite degree.

We answer all three questions in the negative.
We show that the uncountably chromatic graph from an earlier note of ours on a problem of Erd\H{o}s and Hajnal has only ends of countable degree, each dominated by only countably many vertices, yet it has neither an end-faithful nor a topological spanning tree.
A refinement of this graph moreover has no spanning tree that contains a ray from every end of degree~$1$.
\end{abstract}

\maketitle

\section{Introduction}\label{sec_intro}

A rooted spanning tree $S$ of a connected graph $G$ is \emph{end-faithful} if every end of $G$ contains exactly one rooted ray of~$S$, a ray of $S$ that starts at the root. 
Halin~\cite{Halin64} conjectured in 1964 that every connected graph has an end-faithful spanning tree.
The conjecture holds for countable graphs~\cite{Halin64}, for graphs without a subdivided~$K_{\aleph_0}$, which even have normal spanning trees~\cite{Halin78}, and for graphs without a subdivided~$T_{\aleph_1}$, the $\aleph_1$-regular tree~\cite{Polat97}; see~\cite{PitzUnified} and~\cite{BurgerKurkofka} for short proofs of the last two results.
In general, however, Halin's conjecture was refuted by Seymour and Thomas~\cite{SeymourThomas} and, independently, by Thomassen~\cite{Thomassen}.

Both counterexamples are infinitely connected and hence one-ended, and their unique end is large: it has uncountable degree, i.e.\ contains uncountably many disjoint rays, and it is dominated by uncountably many vertices.
Is this the only way in which Halin's conjecture can fail?
Over the following three decades, this question has been asked in several variants.

The most direct form of the question was asked by Diestel:

\begin{problem}[Diestel 1992~{\cite[Problem~4.1]{Diestel92}}]\label{ProbDiestel}
If $G$ is connected and every one-ended connected induced subgraph of $G$ has an end-faithful spanning tree, does $G$ have an end-faithful spanning tree?
\end{problem}

By a theorem of Polat~\cite[Theorem~2.11]{Polat90}, every one-ended connected graph whose end has countable degree has an end-faithful spanning tree.
So a positive answer to \cref{ProbDiestel} would give end-faithful spanning trees in all connected graphs whose ends have countable degree: the rays of a one-ended connected subgraph $H$ of such a graph $G$ all lie in a single end of~$G$, so the end of $H$ has countable degree too.

Diestel and K\"uhn~\cite{DiestelKuhn} proposed a topological version of end-faithful spanning trees.
For a graph $G$, let $|G|$ be the space obtained from $G$, viewed as a 1-complex, by adding its ends as points, with the standard topology \textsc{Top} defined by Diestel and K\"uhn~\cite{DiestelKuhn}; we recall the details in \cref{sec_tst}.
A \emph{topological spanning tree} of $G$ is a path-connected subspace of $|G|$ that contains every vertex and every end of~$G$, contains every edge of which it contains an inner point, and contains no circle~\cite{DiestelKuhn}.
The closure of any normal spanning tree of $G$ is a topological spanning tree, and if a topological spanning tree meets $G$ in a connected subgraph, then this subgraph is an end-faithful spanning tree~\cite[Propositions~7.1 and~7.5]{DiestelKuhn}.
But a topological spanning tree may also be connected through ends, and this extra freedom might be just what is needed to escape the one-ended counterexamples.
Indeed, every one-ended connected graph has a topological spanning tree~\cite[Proposition~7.8]{DiestelKuhn}. 

\begin{problem}[Diestel and K\"uhn 2004~{\cite[Problem~7.7]{DiestelKuhn}}]\label{ProbDK}
Are there connected graphs $G$ such that $|G|$ has no topological spanning tree?
\end{problem}

Carmesin~\cite{Carmesin} took a different route: instead of restricting the graphs, he restricted the ends that the tree has to represent.
A rooted spanning tree $S$ of $G$ is said to be \emph{end-faithful for} a set $\Psi$ of ends of $G$ if every end in $\Psi$ contains exactly one rooted ray of~$S$.
Carmesin proved that every connected graph has a spanning tree that is end-faithful for its undominated ends~\cite[Corollary~2.6]{Carmesin}, and, by induction on $k\in\N$, one that is end-faithful for its ends dominated by at most $k$ vertices~\cite[Theorem~7.1 and Remark~6.14]{Carmesin}.
An end is \emph{thin} if is has finite degree, i.e.\ if it does not contain infinitely many disjoint rays.

\begin{problem}[Carmesin 2019~{\cite[Question~7.2]{Carmesin}}]\label{QCarmesin}
Does every graph have an end-faithful spanning tree for its thin ends?
\end{problem}

Carmesin added that the strengthening of \cref{QCarmesin} to the ends dominated by finitely or by countably many vertices is also open~\cite[p.~52]{Carmesin}.

\subsection*{Our results}
We answer all three problems: the spanning trees in question need not exist.
Our counterexample to \cref{ProbDiestel,ProbDK} is a graph that we had already constructed in 2024, for a problem about uncountable chromatic number~\cite{BowlerPitzChromatic}.
Its definition takes only a few lines:
Consider the injective sequences $t\colon\alpha\to\N$, for countable ordinals~$\alpha$, that are co-infinite, i.e.\ miss infinitely many natural numbers, ordered by $s\le t$ if $s$ is an initial segment of~$t$.
The \emph{immediate predecessor} $s^\star$ of a sequence $s$ of successor length is the initial segment one shorter than~$s$.
An initial segment $s$ of $t$ of successor length is a \emph{record} of~$t$ if its last value is smaller than all values that follow it in~$t$.

\begin{definition}[\cite{BowlerPitzChromatic}]
The graph $\BG$ has these sequences as its vertices, and joins every sequence $t$ to the immediate predecessor $s^\star$ of every record $s$ of~$t$.
\end{definition}

So every sequence of successor length is joined to its immediate predecessor, as it is a record of itself.
A non-empty sequence of limit length has an $\omega$-sequence of records, cofinal in it and with increasing values, and it is joined to their predecessors; see \cref{fig_records}.

\begin{figure}[h!]
\centering
\begin{tikzpicture}[
    val/.style={circle,draw=blue!55!black,fill=white,thin,inner sep=0pt,minimum size=13pt,font=\scriptsize},
    rec/.style={val,draw=red!80!black,thick},
    pre/.style={val,draw=orange!90!black,thick},
    lim/.style={circle,draw=black,fill=white,thin,inner sep=0pt,minimum size=9pt},
    recE/.style={gray!70,semithick},
    gamE/.style={orange!90!black,semithick},
  ]
\def\dx{0.72}
% the branch
\draw (0,0) -- (14*\dx+0.25,0);
\draw[densely dotted] (14*\dx+0.3,0) -- (14*\dx+1.0,0);
\draw (14*\dx+1.05,0) -- (14*\dx+1.55,0);
\node[circle,fill=black,inner sep=0pt,minimum size=2.6pt] at (0,0) {};
% values of t; records (red) and their predecessors (orange)
\foreach \i/\v/\st in {1/7/val,2/12/val,3/3/val,4/10/pre,5/2/rec,6/9/val,7/13/pre,8/4/rec,9/15/val,10/11/pre,11/5/rec,12/14/val,13/16/pre,14/6/rec}
  \node[\st] (n\i) at (\i*\dx,0) {$\v$};
\node[lim,label=right:{$t$}] (t) at (14*\dx+1.55,0) {};
% record graph: t joined to its records (above)
%\foreach \i [evaluate=\i as \h using 0.35+0.14*(14-\i)] in {5,8,11,14}
%  \draw[recE] (t.north) .. controls ($(t.north)+(-0.15,\h)$) and ($(n\i.north)+(0.15,\h)$) .. (n\i.north);
% Gamma: t joined to the predecessors of its records (below)
\foreach \i [evaluate=\i as \h using -0.35-0.14*(14-\i)] in {4,7,10,13}
  \draw[gamE] (t.south) .. controls ($(t.south)+(-0.15,\h)$) and ($(n\i.south)+(0.15,\h)$) .. (n\i.south);
%\node[gray!70!black,anchor=east] at (-0.2,0.75) {\small $\RG$};
\node[orange!80!black,anchor=east] at (-0.2,-0.75) {\small $\BG$};
\end{tikzpicture}
\caption{The records of a limit $t$ (red) and the edges of~$\BG$ at~$t$; nodes are labelled with their last values, and all later values exceed~$6$.}
\label{fig_records}
\end{figure}
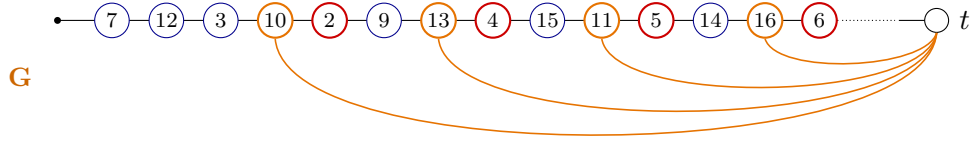

\begin{theorem}\label{GammaThm}
The graph $\BG$ is connected, has $2^{\aleph_0}$ vertices, and:
\begin{enumerate}[(i)]
\item Every end of $\BG$ has countable \emph{combined degree}: it has countable degree and is dominated by only countably many vertices.
\item No spanning tree of $\BG$ contains a ray of every end of~$\BG$.
\item Every path-connected subspace of $|\BG|$ that contains all vertices and ends of $\BG$ contains a circle.
\end{enumerate}
\end{theorem}

Hence, the graph $\BG$ answers \cref{ProbDiestel}, as pointed out above, and \cref{ProbDK}, and countable combined degree does not imply the existence of an end-faithful spanning tree.

However, the graph $\BG$ does not answer Carmesin's \cref{QCarmesin}: every end of $\BG$ has degree $\aleph_0$ and is dominated by infinitely many vertices, so $\BG$ has no thin ends at all.
Our second result refines the construction of $\BG$ so that every spanning tree misses an end of degree~$1$ that is finitely dominated.
The definition is only slightly longer: 
Now every entry of a sequence carries a bit: consider the sequences $t\colon\alpha\to\N\times\{0,1\}$, for countable ordinals~$\alpha$, that are injective and co-infinite in their first coordinates, the \emph{values}; the second coordinates are the \emph{bits}.
Records are defined via the values as before, and a record is a \emph{$0$-record} if its last bit is~$0$.
Let $\Tb$ be the set of those sequences all of whose initial segments of limit length, other than the empty one, have infinitely many $0$-records.

\begin{definition}
The graph $\G$ has vertex set~$\Tb$; it joins every sequence of successor length to its immediate predecessor, and every sequence~$t$ to the immediate predecessor $s^\star$ of every $0$-record $s$ of~$t$.
\end{definition}

So $\G$ is defined like~$\BG$, except that only the records with bit~$0$ count.

\begin{theorem}\label{MainThm}
The graph $\G$ is connected, and:
\begin{enumerate}[(i)]
\item Every end of $\G$ has countable combined degree.
\item Every spanning tree of $\G$ misses an end of~$\G$ of degree~$1$ that is dominated by at least one but only finitely many vertices.
\item Every path-connected subspace of $|\G|$ that contains all vertices and ends of $\G$ contains a circle.
\end{enumerate}
\end{theorem}

As a consequence, the graph $\G$ answers all three problems simultaneously, as well as Carmesin's strengthening of \cref{QCarmesin} to the ends dominated by finitely many vertices.

As every tree in a connected graph extends to a spanning tree, statement~(ii) of \cref{GammaThm,MainThm} holds for all trees, spanning or not.

\subsection*{History and acknowledgements}
The graph $\BG$ has a curious history.
We first studied the end-faithful and the topological spanning tree problems in the summer of 2023, in unpublished notes, on the \emph{record graph}~$\RG$.
Its vertices are all injective sequences $t\colon\alpha\to\N$, for countable ordinals~$\alpha$; it joins every sequence of successor length to its immediate predecessor, and every sequence of limit length to its records themselves, rather than to their immediate predecessors as in~$\BG$.
A diagonal argument showed that no \emph{regressive} spanning tree of~$\RG$, a tree built from one down-neighbour for every vertex other than the root, contains a ray of every end; but $\RG$ turned out to have end-faithful spanning trees of a different kind, and we set the problem aside.
Later, early in 2024, wondering whether graphs of this kind could also be used for a problem of Erd\H{o}s and Hajnal on uncountably chromatic graphs, we had the idea of joining every sequence to the predecessors of its records instead.
This raises the chromatic number from countable to uncountable~\cite[Remark~(2)]{BowlerPitzChromatic}, and it gave~---\ word for word~---\ the graph~$\BG$~\cite{BowlerPitzChromatic}.
We never returned to end-faithful spanning trees for~$\BG$, rashly assuming that so small a change could make no difference there.

As \cref{GammaThm} shows, this was a mistake.
It was Claude (Anthropic) that observed that our published graph $\BG$ has no end-faithful and no topological spanning tree, following closely the arguments of our 2023 notes for regressive spanning trees of~$\RG$, adapted now to $\BG$.
A first version of~$\G$, built from our input~$\RG$, was found by Claude; the current version, based on~$\BG$, was then built by us.
A first version of this paper was assembled by Claude from our 2023 notes, from~\cite{BowlerPitzChromatic} and from~\cite{KurkofkaPitz}; we then rewrote it, in part with Claude's help, shortening the new arguments and isolating more clearly the crucial steps.

\section{Ends of \texorpdfstring{$T$}{T}-graphs}\label{sec_prelims}

We use the terminology of~\cite{KurkofkaPitz}, and we write $\N=\{1,2,3,\dots\}$.
A \emph{ray} is a one-way infinite path, and its subrays are its \emph{tails}.
Two rays of a graph $G$ are \emph{equivalent} if no finite set of vertices separates them; the equivalence classes are the \emph{ends} of~$G$.
For a finite set $X\subseteq V(G)$ and an end~$\eps$, all rays of $\eps$ have tails in the same component $C(X,\eps)$ of $G-X$.
The \emph{degree} of an end is the supremum of the cardinalities of sets of disjoint rays in it, and an end is \emph{thin} if it contains no infinite set of disjoint rays.
A vertex $v$ \emph{dominates} a ray $R$ if there are infinitely many $v$--$(R-v)$ paths that are disjoint except at~$v$; this depends only on the end of~$R$.
An end is \emph{undominated} if no vertex dominates it.
The \emph{combined degree} of an end $\eps$ is the cardinal $\deg(\eps)+\dom(\eps)$, where $\dom(\eps)$ is the number of vertices dominating~$\eps$~\cite{PitzBook}.
Let $S\subseteq G$ be a tree with a root~$r$. A \emph{rooted ray} of $S$ is a ray in $S$ that starts at~$r$, and $S$ is \emph{end-faithful} if every end of $G$ contains exactly one rooted ray of~$S$.
Every ray of $S$ has a tail that is a tail of a rooted ray, namely of the path in $S$ from $r$ to the ray followed by a tail of it; so $S$ contains a ray of an end if and only if it contains a rooted ray of that end.
We say that $S$ \emph{misses} an end of $G$ if it contains no ray of that end.

An \emph{order tree} is a partially ordered set $(T,\le)$ with a least element, its \emph{root}, in which every down-closure $\downcl{t}:=\{t'\in T:t'\le t\}$ is well-ordered.
We write $\sdowncl{t}:=\downcl{t}\setminus\{t\}$ for the chain below~$t$, and $\upcl{t}:=\{t'\in T:t\le t'\}$ for the \emph{up-closure} of~$t$; for $s\le t$ we write $[s,t]$, $(s,t)$, \dots\ for intervals.
A node $t$ is a \emph{successor} if $\sdowncl{t}$ has a maximum, its \emph{immediate predecessor}~$t^\star$, and a \emph{limit} otherwise; in particular the root is a limit.
A \emph{high-ray} of $T$ is a non-empty down-closed chain without a maximum.
A \emph{top} of a high-ray $\varrho$ is a node $t\notin\varrho$ with $\sdowncl{t}=\varrho$; it is necessarily a limit.
The order tree is \emph{Hausdorff} if every high-ray has at most one top~\cite{PitzBook}.

A \emph{$T$-graph}~\cite{BrochetDiestel} is a graph $G$ with vertex set $T$ such that the ends of every edge are comparable in $T$ and every node $t$ has a set of neighbours below it that is cofinal in~$\sdowncl{t}$; in particular, every successor is adjacent to its predecessor, and $G$ is connected.
For a set $X$ of vertices of~$G$, or a subgraph, we write $N(X)$ for the set of vertices outside $X$ with a neighbour in~$X$.
A $T$-graph has \emph{finite adhesion} if $N(\upcl{s})$ is finite for every successor~$s$ of $T$~\cite[Lemma~4.2]{KurkofkaPitz}.
We use the following facts about a $T$-graph~$G$; \ref{K1} and~\ref{K2} are due to Brochet and Diestel, and \ref{K4} is due to Kurkofka and the second author.
\begin{enumerate}[(K1)]
\item\label{K1} Every interval $[s,t]$ of $T$ is connected in~$G$~\cite[Section~2]{BrochetDiestel}; see also~\cite[Lemma~2.3(iv)]{KurkofkaPitz}. Hence so is every up-closure $\upcl s$, the union of the intervals $[s,t]$ with $t\ge s$.
\item\label{K2} Every connected subgraph of $G$ has a unique $T$-minimal vertex~\cite[Section~2]{BrochetDiestel}; see also~\cite[Lemma~2.3(ii)]{KurkofkaPitz}. All its vertices lie above this vertex. 
\item\label{K3} $N(\upcl t)\subseteq\sdowncl t$ for every $t\in T$: a neighbour $x\notin\upcl t$ of a vertex $u\in\upcl t$ is comparable with~$u$, so $x<u$; hence $x$ and $t$ lie in the chain $\downcl u$, and $x<t$ as $x\notin\upcl t$.
\end{enumerate}
If $G$ has finite adhesion, then moreover:
\begin{enumerate}[(K1)]\setcounter{enumi}{3}
\item\label{K4} All chains of $T$ are countable~\cite[Corollary~4.7]{KurkofkaPitz}, and every end of $G$ has countable degree~\cite[Corollary~5.2]{KurkofkaPitz}.
\end{enumerate}

The ends of a $T$-graph of finite adhesion correspond to the high-rays of~$T$, as Kurkofka and the second author showed~\cite[Lemma~5.1]{KurkofkaPitz}; see also~\cite[Lemmas~6.4 and~6.5]{KurkofkaPitz}.
For the convenience of the reader, we give a short self-contained account of the facts we need.
Let $R$ be a ray of a $T$-graph~$G$.
Let $t_0$ be the $T$-minimal vertex of~$R$, see~\ref{K2}, and for $n\ge0$ let $t_{n+1}$ be the $T$-minimal vertex of the tail of $R$ after~$t_n$.
Then $t_0<t_1<\cdots$ are vertices of~$R$, and $\varrho(R):=\bigcup_n\downcl{t_n}$ is a high-ray of~$T$ which $R$ meets cofinally.
For every $t\in T$,
\begin{equation}\label{TailCone}
\text{$R$ has a tail in $\upcl t$ if and only if $t\in\varrho(R)$.}
\end{equation}
Indeed, if $t\le t_{n+1}$, then the tail of $R$ after $t_n$ lies in~$\upcl t$; conversely, a tail in $\upcl t$ contains some~$t_n$, and then $t\le t_n$.

\begin{lemma}\label{Ends}
Let $G$ be a $T$-graph of finite adhesion.
\begin{enumerate}[(i)]
\item If $s$ is a successor and $R$ a ray with $s\notin\varrho(R)$, then $R$ has a tail outside $\upcl s\cup N(\upcl s)$.
\item Two rays $R,R'$ of $G$ are equivalent if and only if $\varrho(R)=\varrho(R')$.
\item A ray $R$ of $G$ satisfies $\varrho(R)=\varrho$ for a high-ray $\varrho$ if and only if $V(R)\cap\varrho$ is cofinal in~$\varrho$; and every high-ray of $T$ is of the form $\varrho(R)$ for some ray~$R$.
\end{enumerate}
\end{lemma}

\begin{proof}
(i) The up-closure $\upcl s$ is connected, so it spans a component of $G-N(\upcl s)$; hence a tail of $R$ avoiding the finite set $N(\upcl s)$ lies in $\upcl s$ or avoids it, and by~\eqref{TailCone} it avoids~$\upcl s$.

(ii) If $\varrho(R)\ne\varrho(R')$, we may assume that some $t\in\varrho(R)$ does not lie in~$\varrho(R')$; then the successor $s$ of $t$ in $\varrho(R)$ does not lie in $\varrho(R')$ either.
By~\eqref{TailCone} and~(i), $R$ has a tail in $\upcl s$ and $R'$ has a tail outside $\upcl s\cup N(\upcl s)$, so the finite set $N(\upcl s)$ separates them.
Conversely, let $\varrho:=\varrho(R)=\varrho(R')$ and let $X$ be a finite set of vertices.
Pick $t\in\varrho$ above all of $X\cap\varrho$; the nodes of $\varrho$ above~$t$ form a connected subgraph of $G-X$ by~\ref{K1}, and every tail of $R$ or of $R'$ meets it, as both rays meet $\varrho$ cofinally.
So $X$ does not separate $R$ from~$R'$.

(iii) If $\varrho(R)=\varrho$, then $R$ meets $\varrho$ cofinally.
Conversely, suppose that $V(R)\cap\varrho$ is cofinal in~$\varrho$.
For a successor $s\in\varrho$, infinitely many vertices of $R$ lie in~$\upcl s$, so $s\in\varrho(R)$ by~(i); as the successors are cofinal in~$\varrho$, we get $\varrho\subseteq\varrho(R)$.
If $\varrho(R)$ contained a node above all of~$\varrho$, it would contain a successor $s'$ above all of~$\varrho$, and by~\eqref{TailCone} a tail of $R$ would lie in $\upcl{s'}$, which is disjoint from~$\varrho$, while $R$ meets $\varrho$ infinitely often.
So $\varrho(R)=\varrho$.
Finally, a high-ray $\varrho$ has a cofinal sequence $u_0<u_1<\cdots$ by~\ref{K4}, and by~\ref{K1} there are $u_n$--$u_{n+1}$ paths in the intervals $[u_n,u_{n+1}]$; their union is a ray with all its vertices in~$\varrho$, which meets $\varrho$ cofinally.
\end{proof}

By \cref{Ends}(ii), all rays of an end $\eps$ of $G$ have the same high-ray $\varrho(\eps):=\varrho(R)$, and by \cref{Ends}(ii) and~(iii), every high-ray $\varrho$ of $T$ is $\varrho(\eps)$ for exactly one end $\eps=:\eps_\varrho$, the end \emph{belonging} to~$\varrho$.

\begin{lemma}\label{Dominators}
Let $G$ be a $T$-graph of finite adhesion, let $\varrho$ be a high-ray, and $v$ a vertex dominating~$\eps_\varrho$.
Then $v$ lies in $\varrho$ or is a top of~$\varrho$.
Moreover, if $v\in\varrho$, then $v\in N(\upcl{s})$ for every successor $s\in\varrho$ above~$v$.
\end{lemma}

\begin{proof}
As $v$ dominates the rays of~$\eps_\varrho$, no finite set of vertices other than $v$ separates $v$ from a tail of such a ray.
Let $s\in\varrho$ be a successor with $v\notin\upcl{s}$.
By~\eqref{TailCone}, the rays of $\eps_\varrho$ have tails in~$\upcl{s}$, and the finite set $N(\upcl{s})$ separates these tails from $v$ unless $v\in N(\upcl{s})$; so $v\in N(\upcl{s})\subseteq\varrho$ by~\ref{K3}.
This proves the second assertion, and the first if such an $s$ exists.
Otherwise $v$ lies above all of $\varrho$ but not in~$\varrho$, so the least element $t$ of $\downcl{v}\setminus\varrho$ is a top of~$\varrho$.
If $t<v$, then the successor $s'$ of $t$ in $\downcl{v}$ does not lie in~$\varrho$; so by \cref{Ends}(i) the rays of $\eps_\varrho$ have tails outside $\upcl{s'}\cup N(\upcl{s'})$, and $N(\upcl{s'})$ separates these tails from $v\in\upcl{s'}$.
So $v=t$.
\end{proof}

So if $T$ is Hausdorff, then every end of a $T$-graph of finite adhesion has countable combined degree, by \cref{Dominators} and~\ref{K4}.

\section{The graph \texorpdfstring{$\BG$}{G}}\label{sec_Gamma}

We now repeat the definition of $\BG$ from the introduction formally.
Let $T^*$ be the set of all injective transfinite sequences $t\colon\alpha\to\N$ with $\alpha<\omega_1$, ordered by extension, i.e.\ $t\le t'$ if $t=t'\restriction\alpha$ for the domain $\alpha$ of~$t$.
This is an order tree; its root is the empty sequence.
For a sequence $s\colon\alpha+1\to\N$ of successor length, let $\last(s):=s(\alpha)\in\N$ be the last value of~$s$, and $s^\star:=s\restriction\alpha$ its immediate predecessor; these sequences are the successor nodes of~$T^*$, and $\Sigma(T^*)$ is their set.
We write $\im(t)$ for the set of values of~$t$, and for $n\in\N\setminus\im(t)$ the sequence $t\conc n$ is $t$ followed by~$n$.
A node $t$ is \emph{co-infinite} if $|\N\setminus\im(t)|=\infty$, and $\Tc\subseteq T^*$ is the down-closed subtree of all co-infinite nodes; we put $\Sigma(\Tc):=\Sigma(T^*)\cap\Tc$.
So $\Tc$ is Hausdorff: a high-ray $\varrho$ of $\Tc$ has at most one top in~$\Tc$, the sequence $\bigcup\varrho$, and it has one if and only if this sequence is co-infinite, because a sequence of limit length is the union of its proper initial segments.

\begin{definition}[Records]
For $t\in T^*$, let
\[
A_t:=\{s\le t:s\in\Sigma(T^*),\ \last(s)=\min(\im(t)\setminus\im(s^\star))\},
\]
and $A^\star_t:=\{s^\star:s\in A_t\}$, as in~\cite[Section~2]{BowlerPitzChromatic}.
We call the elements of $A_t$ the \emph{records} of~$t$: they are the nodes $s\le t$ in $\Sigma(T^*)$ whose last value is the minimum of the tail of $t$ beginning at~$s$, that is, $\last(s)<\last(y)$ for every $y\in\Sigma(T^*)$ with $s<y\le t$.
\end{definition}

Every $s\in\Sigma(T^*)$ is a record of itself, and the definition of $A_s$ implies that for all $s\le u\in T^*$ we have
\[
A_u\cap\sdowncl{s}\subseteq A_s\qquad\text{and}\qquad A^\star_u\cap\sdowncl{s}\subseteq A^\star_s.\tag{$\star$}
\]

\begin{lemma}\label{Records}
Let $t\in T^*$ be a node other than the root.
\begin{enumerate}[(i)]
\item If $s<s'$ are records of $t$, then $\last(s)<\last(s')$.
\item If $t\in\Sigma(T^*)$, then $|A^\star_t|=|A_t|\le\last(t)$.
\item If $t$ is a limit, then for every $z<t$, the node $y\in\Sigma(T^*)$ with $z<y<t$ where $\last(y)$ is minimal is a record of~$t$. So the records of a limit $t$ are cofinal in $\sdowncl{t}$; they form an $\omega$-sequence, and only finitely many of them lie below any given node $z<t$.
\end{enumerate}
\end{lemma}

\begin{proof}
(i) is the definition for $y=s'$.
(ii) Every record $s$ of $t$ satisfies $\last(s)\le\last(t)$, with equality only for $s=t$, so by~(i) the records of $t$ have distinct last values in $\{1,\dots,\last(t)\}$; and distinct records have distinct immediate predecessors.
(iii) The nodes $y\in\Sigma(T^*)$ with $z<y<t$ have pairwise distinct last values, so one of them has the least last value, and it is a record of $t$ because every $y'\in\Sigma(T^*)$ with $y<y'\le t$ also lies in~$(z,t)$.
By~(i), the last values of the records strictly increase, so the records form a sequence of order type at most~$\omega$, and of order type exactly $\omega$ as they are cofinal.
A record below $z$ has a smaller last value than this node~$y$, so there are only finitely many of them.
\end{proof}

\begin{definition}[\cite{BowlerPitzChromatic}]
The graph $\BG$ has vertex set $V(\BG)=\Tc$ and edge set $E(\BG)=\{t't:t\in\Tc,\ t'\in A^\star_t\}$.
\end{definition}

So the neighbours of a node $t$ below it are the elements of~$A^\star_t$; the immediate predecessors of the records of a limit may themselves be limits.
\begin{lemma}[{\cite[Remark~(1) and~(${\star}$)]{BowlerPitzChromatic}}]\label{GammaTGraph}
The graph $\BG$ is a $T$-graph of finite adhesion on~$\Tc$.
More precisely, $N(\upcl q)\subseteq A^\star_q$ for every $q\in\Sigma(\Tc)$, and $|A^\star_q|\le\last(q)$.
\end{lemma}

\begin{proof}
The ends of every edge are comparable.
Every $t\in\Sigma(\Tc)$ is adjacent to $t^\star$, as $t\in A_t$, and every limit $\mu$ other than the root is adjacent to the elements of $A^\star_\mu$, which are cofinal in $\sdowncl\mu$ by \cref{Records}(iii).
So $\BG$ is a $T$-graph.
Now let $q\in\Sigma(\Tc)$, and let $u\in\upcl q$ be adjacent to $x\notin\upcl q$.
Then $x<q$ by~\ref{K3}, and $x<u$, so $x\in A^\star_u\cap\sdowncl q\subseteq A^\star_q$ by~$(\star)$.
Finally, $|A^\star_q|\le\last(q)$ by \cref{Records}(ii).
\end{proof}

\begin{lemma}\label{GammaEnds}
Every end of $\BG$ has combined degree~$\aleph_0$: it has degree $\aleph_0$ and is dominated by infinitely, but only countably, many vertices.
\end{lemma}

\begin{proof}
Every end has countable combined degree by \cref{Dominators} and~\ref{K4}, as $\Tc$ is Hausdorff.
Now let $\varrho$ be a high-ray of~$\Tc$.
The sequence $u:=\bigcup\varrho$ is a limit of $T^*$ other than the root, and its records $x_1<x_2<\cdots$ lie in $\varrho$ and are cofinal in it, by \cref{Records}(iii).
For $i<j$ we have $x_i^\star<x_j^\star\le u$, so $x_i^\star\in A^\star_u\cap\sdowncl{x_j^\star}\subseteq A^\star_{x_j^\star}$ by~$(\star)$, and $x_i^\star x_j^\star$ is an edge of~$\BG$.
So the vertices $x_1^\star<x_2^\star<\cdots$ span an infinite clique, and they are cofinal in~$\varrho$; see \cref{fig_clique}.
Partition them into infinitely many infinite sets; each of them spans a ray whose vertex set is contained and cofinal in~$\varrho$, and these rays are disjoint and belong to $\eps_\varrho$ by \cref{Ends}(iii).
Every $x_i^\star$ is adjacent to all vertices of these rays, so it dominates~$\eps_\varrho$.
\end{proof}

\begin{figure}[h!]
\centering
\begin{tikzpicture}[
    val/.style={circle,draw=blue!55!black,fill=white,thin,inner sep=0pt,minimum size=13pt,font=\scriptsize},
    rec/.style={val,draw=red!80!black,thick},
    pre/.style={val,draw=orange!90!black,thick},
    gamE/.style={orange!90!black,semithick},
  ]
\def\dx{0.72}
\draw (0,0) -- (14*\dx+0.25,0);
\draw[densely dotted] (14*\dx+0.3,0) -- (14*\dx+1.0,0);
\draw[->] (14*\dx+1.05,0) -- (14*\dx+1.6,0) node[right] {$\varepsilon_\varrho$};
\node[circle,fill=black,inner sep=0pt,minimum size=2.6pt] at (0,0) {};
\foreach \i/\v/\st in {1/7/val,2/12/val,3/3/val,4/10/pre,5/2/rec,6/9/val,7/13/pre,8/4/rec,9/15/val,10/11/pre,11/5/rec,12/14/val,13/16/pre,14/6/rec}
  \node[\st] (n\i) at (\i*\dx,0) {$\v$};
% the clique on the predecessors of the records
\foreach \a/\b [evaluate=\a as \h using 0.30+0.11*(\b-\a)] in {4/7,4/10,4/13,7/10,7/13,10/13}
  \draw[gamE] (n\a.north) .. controls ($(n\a.north)+(0.2,\h)$) and ($(n\b.north)+(-0.2,\h)$) .. (n\b.north);
% edges continue towards later predecessors
\foreach \a [evaluate=\a as \h using -0.35-0.06*(14-\a), evaluate=\a as \xe using 14*0.72+1.0] in {4,7,10,13}
  \draw[gamE,densely dashed] (n\a.south) .. controls ($(n\a.south)+(0.3,\h)$) and (\xe,\h) .. (\xe+0.2,-0.15);
\end{tikzpicture}
\caption{The predecessors of the records of $\bigcup\varrho$ span an infinite clique in~$\BG$.}
\label{fig_clique}
\end{figure}
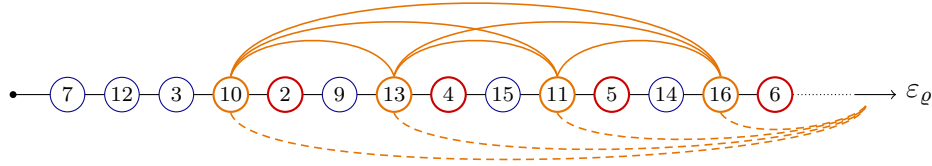

Finally, $\BG$ has $2^{\aleph_0}$ vertices: $\Tc$ contains the $2^{\aleph_0}$ injective sequences of length~$\omega$ with even values, and $T^*$ has at most $\aleph_1\cdot2^{\aleph_0}=2^{\aleph_0}$ elements.

\section{No spanning tree contains a ray of every end}\label{sec_proof}

Throughout this section, $S$ is a spanning tree of~$\BG$, which we regard as rooted at the root $r$ of~$\Tc$; so for every vertex $v$ there is a unique path $vSr$ in $S$ from $v$ to~$r$.
We find an end of $\BG$ that contains no ray of~$S$.

\begin{definition}[Exits]
For a vertex $v\ne r$, the \emph{exit} $h(v)$ of $v$ is the first vertex of the path $vSr$ that does not lie in the up-closure~$\upcl{v}$.
\end{definition}

\begin{lemma}\label{Exit}
For every vertex $v\ne r$ we have $h(v)<v$.
\end{lemma}

\begin{proof}
The vertex $u$ of $vSr$ preceding $h(v)$ lies in $\upcl{v}$ and is adjacent to $h(v)\notin\upcl{v}$.
So $h(v)\in\sdowncl{v}$ by~\ref{K3}.
\end{proof}

The exit tells us how $S$ leaves the up-closure of $v$ on its way to the root.
Inside $\upcl v$, the path $vSr$ may wander up and down arbitrarily; but it has to leave $\upcl v$ eventually, and when it does, it drops to a node below~$v$.
If $S$ were obtained by choosing one down-neighbour for every vertex other than the root, then $h(v)$ would simply be the chosen down-neighbour of~$v$; for arbitrary trees, $h$ plays the same role; see \cref{fig_exit}.

\begin{figure}[h!]
\centering
\begin{tikzpicture}[
    dot/.style={circle,fill=blue!55!black,inner sep=0pt,minimum size=3.4pt},
    path/.style={green!45!black,line width=1.1pt},
    exitE/.style={orange!90!black,line width=1.1pt},
  ]
% the chain below v and the up-closure of v
\fill[blue!6] (3,0) -- (6.4,-2.0) -- (6.4,2.0) -- cycle;
\draw[blue!40] (6.4,-2.0) -- (3,0) -- (6.4,2.0);
\node[blue!60!black] at (5.95,1.2) {$\upcl v$};
\draw (0,0) -- (3,0);
% nodes
\node[dot,label=below:{$r$}] (r) at (0,0) {};
\node[dot,label={[yshift=1pt]above:{$h(v)$}}] (h) at (1.5,0) {};
\node[dot] (h2) at (0.8,0) {};
\node[dot,label=above left:{$v$}] (v) at (3,0) {};
\node[dot] (a) at (3.8,0.3) {};
\node[dot] (b) at (4.7,0.75) {};
\node[dot] (c) at (5.5,0.3) {};
\node[dot] (d) at (4.6,-0.3) {};
\node[dot] (u) at (5.4,-0.95) {};
% the path vSr: wanders in the up-closure, then exits along one edge
\draw[path] (v) -- (a) -- (b) -- (c) -- (d) -- (u);
\draw[exitE] (u) .. controls (4.2,-2.6) and (1.9,-1.6) .. (h);
\draw[path,densely dashed] (h) -- (h2) -- (r);
\node[green!35!black] at (4.05,1.55) {$vSr$};
\end{tikzpicture}
\caption{The first vertex of the path $vSr$ outside the up-closure of $v$ is its exit~$h(v)$, which lies below~$v$.}
\label{fig_exit}
\end{figure}
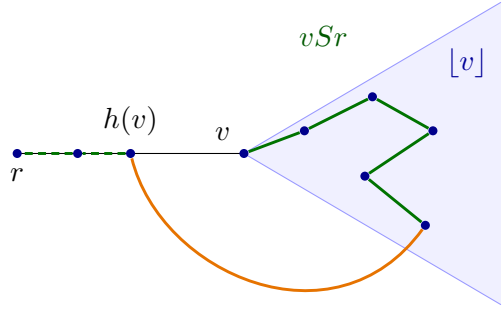

\subsection*{Extensions}
For $t\in\Sigma(\Tc)$, we say $t'\in\Tc$ is an \emph{extension of~$t$} if $t'\ge t$ and $\im(t')\setminus\im(t)\subseteq\{n\in\N:n>\last(t)\}$.
We say $t'\in\Tc$ is a \emph{$1$-extension of~$t$} if it has the stronger property that $t'>t$ and, letting $a_1$ be the minimal element of $\N\setminus\im(t)$ with $a_1>\last(t)$, we have $\im(t')\setminus\im(t)\subseteq\{n\in\N:n>a_1\}$.
In this case we also say that the $1$-extension $t'$ \emph{skips}~$a_1$.
These notions are taken from~\cite[Section~3]{BowlerPitzChromatic}, where two diagonal arguments of the same kind show that $\BG$ has uncountable chromatic number.

Skipped values are reserves.
If $t'$ is a $1$-extension of $t$ skipping~$a_1$, then $a_1$ is still unused at~$t'$, and it is smaller than all values of $t'$ after~$t$.
Appending it later creates a node whose last value is a new minimum: it becomes a record of every node above it that adds only larger values.
The first diagonal argument below keeps such reserves unused forever, and the second appends them.

\subsection*{The first diagonal}
The claim below and the two subsequent diagonal arguments go back to our unpublished notes from 2023, where they show that no spanning tree of~$\RG$ built from one down-neighbour for every vertex other than the root contains a ray of every end.

\begin{quote}
\emph{Claim.} Every $t$ in $\Sigma(\Tc)$ has an extension $t'$ in $\Sigma(\Tc)$ such that every $1$-extension $t''$ of $t'$ that is a limit satisfies $h(t'')>t$.
\end{quote}
This is the claim of~\cite[Section~3]{BowlerPitzChromatic}, with $h(t'')>t$ in place of $c(t'')>c(t^\star)$ for a colouring~$c$, and only for those $1$-extensions $t''$ that are limits.
As $h(t'')$ and $t$ both lie in the chain $\sdowncl{t''}$, the negation of $h(t'')>t$ is $h(t'')\le t$.

\begin{lemma}\label{FirstDiagonal}
The claim is false: there exists $t_0\in\Sigma(\Tc)$ such that all its extensions $t'\in\Sigma(\Tc)$ have a $1$-extension $t''$ that is a limit with $h(t'')\le t_0$.
\end{lemma}

\begin{proof}
Suppose for a contradiction that the claim holds.
Fix an arbitrary $s_0\in\Sigma(\Tc)$.
Let $s'_0\in\Sigma(\Tc)$ be an extension of $s_0$ as in the claim.
Let $a_1<a_2$ be the two smallest elements of $\N\setminus\im(s'_0)$ above $\last(s'_0)$.
Let $s_1:=s'_0\conc a_2$.
Let $s'_1$ be an extension of $s_1$ as in the claim.
Let $a_3<a_4$ be the two smallest elements of $\N\setminus\im(s'_1)$ above $\last(s'_1)$.
Let $s_2:=s'_1\conc a_4$.
And so on.

Then $\hat s:=\bigcup_n s_n$ is an injective sequence of limit length, and all its values after $s'_n$ are at least $a_{2n+2}$, as $s'_{n+1}$ is an extension of $s_{n+1}=s'_n\conc a_{2n+2}$ and $a_{2n+4}>\last(s'_{n+1})$.
Moreover, $a_1,a_3,a_5,\dots$ witness that $\hat s$ is co-infinite, giving $\hat s\in\Tc$.
But for each~$n$, the limit $\hat s$ is a $1$-extension (skipping $a_{2n+1}$) of the extension $s'_n$ of~$s_n$, so $h(\hat s)>s_n$ according to the claim.
However, $h(\hat s)<\hat s$ by \cref{Exit}, so $h(\hat s)\le s_n$ for some~$n$, a contradiction.
\end{proof}

From now on, we fix $t_0$ as in \cref{FirstDiagonal}.
So arbitrarily high up above $t_0$ there are limits $t''$ such that the path $t''Sr$, when it leaves the up-closure of~$t''$, drops to a vertex at or below~$t_0$.

\subsection*{The second diagonal}
We now build a high-ray through limits as in \cref{FirstDiagonal}, joined by \emph{boundaries} that carry small values.
For the extension $t'_0=t_0$ of~$t_0$, there is a $1$-extension $t''_0$ of $t'_0$ skipping $a_1\in\N$ that is a limit with $h(t''_0)\le t_0$.
Let $t'_1:=t''_0\conc a_1$.
Then $t'_1\in\Sigma(\Tc)$ is itself an extension of~$t_0$ (\cref{Bookkeeping}), so it has a $1$-extension $t''_1$ skipping $a_2$ that is a limit with $h(t''_1)\le t_0$.
Let $t'_2:=t''_1\conc a_2$.
And so on.
We call the nodes $t'_1<t'_2<\cdots$ the \emph{boundaries}, and for $n\ge0$ the nodes $y\in\Sigma(\Tc)$ with $t'_n<y<t''_n$ form the \emph{$n$th stage}.

\begin{lemma}\label{Bookkeeping}
The recursion is well defined, and for every $n\ge0$:
\begin{enumerate}[(i)]
\item every node $y$ in the $n$th stage satisfies $\last(y)>a_{n+1}$;
\item $\last(t_0)<a_{n+1}<a_{n+2}$;
\item $t'_{n+1}\in\Sigma(\Tc)$ is an extension of $t_0$, with $\last(t'_{n+1})=a_{n+1}$.
\end{enumerate}
\end{lemma}

\begin{proof}
We use induction on~$n$; suppose that $t'_n\in\Sigma(\Tc)$ is an extension of~$t_0$ with $\last(t'_n)\ge\last(t_0)$, as is the case for $t'_0=t_0$.
Then $a_{n+1}$ and $t''_n$ are defined.

(i) The value $\last(y)$ is a value of $t''_n$ that is not a value of~$t'_n$, and $t''_n$ is a $1$-extension of $t'_n$ skipping~$a_{n+1}$.

(ii) We have $a_{n+1}>\last(t'_n)\ge\last(t_0)$ by the choice of~$a_{n+1}$; once (iii) is shown, $a_{n+2}>\last(t'_{n+1})=a_{n+1}$ in the same way.

(iii) The value $a_{n+1}$ is not a value of~$t'_n$, and by~(i) it is not a value added in the $n$th stage either; so $t'_{n+1}$ is an injective sequence, with $\last(t'_{n+1})=a_{n+1}$, and it is co-infinite as $t''_n$ is.
Every value of $t'_{n+1}$ after $t_0$ is a value added in a stage $k\le n$, which exceeds $a_{k+1}>\last(t_0)$ by~(i) and~(ii), or a boundary value $a_{k+1}>\last(t_0)$ with $k\le n$.
So $t'_{n+1}$ is an extension of~$t_0$ with $\last(t'_{n+1})\ge\last(t_0)$, which completes the induction.
\end{proof}

The remaining two lemmas of this section, which show that $S$ misses an end, were found by Claude and rewritten by us (see Figure~\ref{fig_diagonal} for an illustration).

\begin{lemma}[Bottlenecks]\label{Bottleneck}
For every $n\ge0$, we have $N(\upcl{t'_{n+1}})\subseteq A^\star_{t'_{n+1}}\subseteq\{t''_0,\dots,t''_n\}\cup A^\star_{t_0}$, and $A^\star_{t_0}\subseteq\sdowncl{t_0}$ is finite.
\end{lemma}

\begin{figure}[h!]
\centering
\begin{tikzpicture}[
    xscale=1.0,
    dot/.style={circle,fill=blue!55!black,inner sep=0pt,minimum size=3.4pt},
    bnd/.style={circle,fill=red!80!black,inner sep=0pt,minimum size=4.6pt},
    lim/.style={circle,draw=black,fill=white,thin,inner sep=0pt,minimum size=5pt},
    sq/.style={rectangle,fill=blue!55!black,inner sep=0pt,minimum size=3.6pt},
    edgeE/.style={orange!90!black,semithick},
    exitE/.style={orange!90!black,semithick,densely dashed},
    noE/.style={gray!70,densely dotted,thick},
  ]
% clip to the drawing: the control points of the curves would enlarge the bounding box
\clip (-0.25,-2.35) rectangle (13.5,2.55);
% the cone above the boundary t'_2
\fill[red!7] (7.5,0) -- (13.2,2.0) -- (13.2,-2.0) -- cycle;
\draw[red!35] (13.2,2.0) -- (7.5,0) -- (13.2,-2.0);
\node[red!60!black] at (12.7,1.2) {$\upcl{t'_2}$};
% the branch, with breaks before the limits
\draw (0,0) -- (3.55,0);
\draw (3.85,0) -- (6.85,0);
\draw (7.15,0) -- (10.15,0);
\draw[->] (10.45,0) -- (13.0,0) node[right] {$\eps$};
\foreach \bx in {3.7,7.0,10.3}{
  \draw ($(\bx,0)+(-0.10,-0.14)$) -- ($(\bx,0)+(0.02,0.14)$);
  \draw ($(\bx,0)+(-0.02,-0.14)$) -- ($(\bx,0)+(0.10,0.14)$);
}
% root, the finite set A*_{t_0}, t_0
\node[dot,label=below:{$r$}] at (0,0) {};
\node[sq] (A1) at (0.55,0) {};
\node[sq] (A2) at (1.0,0) {};
\node[dot,label={[yshift=1pt]above:{$t_0$}}] (t0) at (1.5,0) {};
% stage nodes
\foreach \x in {2.0,2.5,3.0,3.4, 4.6,5.1,5.6,6.1,6.6, 7.9,8.4,8.9,9.4,9.9, 11.2,11.7,12.2}
  \node[dot] at (\x,0) {};
% limits t''_n and boundaries t'_{n+1}
\node[lim,label={[yshift=1pt]above:{$t''_0$}}] (l0) at (3.95,0) {};
\node[bnd,label={[yshift=-1pt]below:{$t'_1$}}] (b1) at (4.2,0) {};
\node[lim,label={[yshift=1pt]above left:{$t''_1$}}] (l1) at (7.25,0) {};
\node[bnd,label={[yshift=-1pt]below:{$t'_2$}}] (b2) at (7.5,0) {};
\node[lim,label={[yshift=1pt]above:{$t''_2$}}] (l2) at (10.55,0) {};
\node[bnd,label={[yshift=-1pt]below:{$t'_3$}}] (b3) at (10.8,0) {};
% stage labels
\node[gray] at (2.45,0.42) {\scriptsize stage $0$};
\node[gray] at (5.6,0.42) {\scriptsize stage $1$};
\node[gray] at (8.9,0.42) {\scriptsize stage $2$};
% edges into the cone above t'_2: from earlier limits and from A*_{t_0} only
\draw[edgeE] (l0) .. controls (6.0,2.4) and (9.2,2.4) .. (l2);
\draw[edgeE] (l1) .. controls (8.3,1.35) and (9.6,1.35) .. (l2);
\draw[edgeE] (A1) .. controls (3.0,3.4) and (9.0,3.1) .. (l2);
% exits of the limits t''_n drop to t_0 or below
\draw[exitE] (l0) .. controls (3.2,-1.0) and (2.0,-1.0) .. (t0);
\draw[exitE] (l1) .. controls (5.5,-2.0) and (1.6,-2.0) .. (A2);
\draw[exitE] (l2) .. controls (8.0,-3.0) and (1.4,-3.0) .. (A1);
\node[orange!80!black] at (5.6,-1.85) {\scriptsize exits $h(t''_n)\le t_0$};
\end{tikzpicture}
\caption{Bottlenecks in~$\BG$: the up-closure of a boundary $t'_{n+1}$ can be entered only from the limits $t''_0,\dots,t''_n$ or from the finite set $A^\star_{t_0}$ (squares).}
\label{fig_diagonal}
\end{figure}
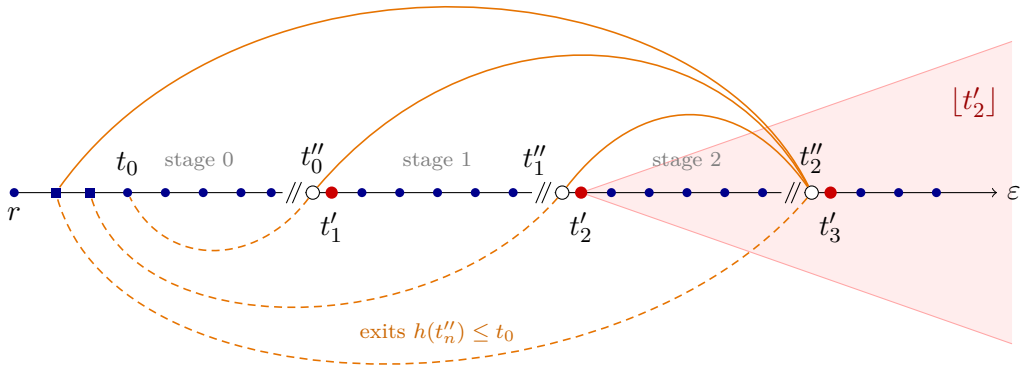

\begin{proof}
The first inclusion is \cref{GammaTGraph}.
Let $y\in A_{t'_{n+1}}$.
If $y^\star<t_0$, then $y^\star\in A^\star_{t_0}$ by~$(\star)$.
Otherwise $y$ is a node of $\Sigma(\Tc)$ in $(t_0,t'_{n+1}]$, so it is a boundary $t'_{k+1}$ with $k\le n$, and then $y^\star=t''_k$, or it lies in the $k$th stage for some $k\le n$.
In the latter case $\last(y)>a_{k+1}=\last(t'_{k+1})$ by \cref{Bookkeeping}, while $y<t'_{k+1}\le t'_{n+1}$; so $y\notin A_{t'_{n+1}}$.
\end{proof}

Let $\hat\varrho:=\bigcup_{n\ge1}\downcl{t'_n}$, a high-ray of $\Tc$ containing $t_0$, all $t''_n$ and all boundaries $t'_n$, and let $\eps:=\eps_{\hat\varrho}$ be its end.

\begin{lemma}\label{Missed}
No rooted ray of $S$ belongs to~$\eps$.
\end{lemma}

\begin{proof}
Suppose that a rooted ray $R$ of $S$ belongs to~$\eps$.
By~\eqref{TailCone}, $R$ has a tail $R_1$ in~$\upcl{t'_1}$.

We show that every tail of $R_1$ contains one of the limits~$t''_m$.
Let $R'$ be a tail of~$R_1$.
By \cref{Ends}(iii), $R'$ has a vertex $x\in\hat\varrho$, and $x<t'_{n+1}$ for some~$n$, so $x\notin\upcl{t'_{n+1}}$.
By~\eqref{TailCone}, the ray $xR'$ has a tail in $\upcl{t'_{n+1}}$, so it contains an edge $uy$ with $u\notin\upcl{t'_{n+1}}$ and $y\in\upcl{t'_{n+1}}$.
Then $u\in N(\upcl{t'_{n+1}})\subseteq\{t''_0,\dots,t''_n\}\cup A^\star_{t_0}$ by \cref{Bottleneck}, see \cref{fig_diagonal}, and as $u\in\upcl{t'_1}$ while $A^\star_{t_0}\subseteq\sdowncl{t_0}$, we have $u=t''_m$ for some~$m\le n$.

So $R_1$ contains infinitely many of the limits~$t''_m$; see \cref{fig_twopassages}.
Their indices, in the order in which $R_1$ meets them, cannot decrease forever; so $R_1$ meets some $t''_k$ before some $t''_{k'}$ with $k<k'$.
The path $t''_{k'}Sr$ is the initial segment of $R$ up to $t''_{k'}$, traversed backwards.
It contains the vertex $t''_k\notin\upcl{t''_{k'}}$, so its first vertex outside $\upcl{t''_{k'}}$, the exit $h(t''_{k'})$, comes no later than~$t''_k$; and $h(t''_{k'})\ne t''_k$, as $h(t''_{k'})\le t_0<t''_k$.
So $R$ meets $h(t''_{k'})$ after~$t''_k$, that is, on~$R_1$.
But $R_1\subseteq\upcl{t'_1}$, while $h(t''_{k'})\le t_0<t'_1$, a contradiction.
\end{proof}

\begin{figure}[h!]
\centering
\begin{tikzpicture}[
    dot/.style={circle,fill=blue!55!black,inner sep=0pt,minimum size=3.4pt},
    wdot/.style={circle,fill=green!45!black,inner sep=0pt,minimum size=4.5pt},
    lim/.style={circle,draw=black,fill=white,thin,inner sep=0pt,minimum size=5.5pt},
    rayR/.style={green!45!black,line width=1.2pt},
    exitE/.style={orange!90!black,semithick,densely dashed},
    yscale=0.7
  ]
\fill[red!6] (2.6,0) -- (12.6,3.5) -- (12.6,-3.5) -- cycle;
\draw[red!35] (12.6,3.5) -- (2.6,0) -- (12.6,-3.5);
\fill[red!14] (9.0,0) -- (12.6,2.16) -- (12.6,-2.16) -- cycle;
\draw[red!45] (12.6,2.16) -- (9.0,0) -- (12.6,-2.16);
\node[red!60!black] at (11.9,3.05) {$\upcl{t'_1}$};
\node[red!60!black] at (12.15,-1.25) {$\upcl{t''_{k'}}$};
\draw (0,0) -- (12.0,0);
\draw[->] (12.0,0) -- (12.8,0) node[right] {$\eps$};
\node[dot,label=below:{$r$}] (r) at (0,0) {};
\node[dot,label=below:{$t_0$}] (t0) at (1.8,0) {};
\node[dot,label={[xshift=-3pt]below:{$t'_1$}}] (t1) at (2.6,0) {};
\node[lim,label=below:{$t''_k$}] (lk) at (5.6,0) {};
\node[lim,label={[xshift=-6pt]below:{$t''_{k'}$}}] (lkk) at (9.0,0) {};
\draw[rayR] (r) .. controls (0.8,1.1) and (2.8,1.2) .. (4.2,0.42)
  .. controls (4.8,0.15) and (5.3,0.05) .. (lk);
% early visit to the cone above t''_{k'} (high loop)
\draw[rayR] (lk) .. controls (6.3,0.9) and (8.0,1.9) .. (10.4,2.1)
  .. controls (12.0,2.2) and (12.3,1.1) .. (11.6,0.95)
  .. controls (10.6,0.9) and (9.0,0.95) .. (7.8,0.8);
\node[wdot,label={[green!35!black,font=\small]left:{$h(t''_{k'})$}}] (w) at (7.8,0.8) {};
% last entry and arrival at t''_{k'}
\draw[rayR] (w) .. controls (8.4,0.45) and (9.6,0.4) .. (10.2,0.28)
  .. controls (10.6,0.2) and (9.9,0.04) .. (lkk);
\draw[rayR,->] (lkk) .. controls (9.8,-0.55) and (11.0,-0.75) .. (12.3,-0.7);
\node[green!35!black] at (0.8,1.15) {$R$};
\draw[exitE] (lkk) .. controls (7.0,-2.2) and (3.0,-2.2) .. (t0);
\node[orange!80!black,fill=white,inner sep=1pt] at (5.4,-1.75) {\small $h(t''_{k'})\le t_0$};
\end{tikzpicture}
\caption{If $R$ meets $t''_k$ before $t''_{k'}$, then the exit $h(t''_{k'})$ lies in~$\upcl{t'_1}$, not at or below~$t_0$.}
\label{fig_twopassages}
\end{figure}
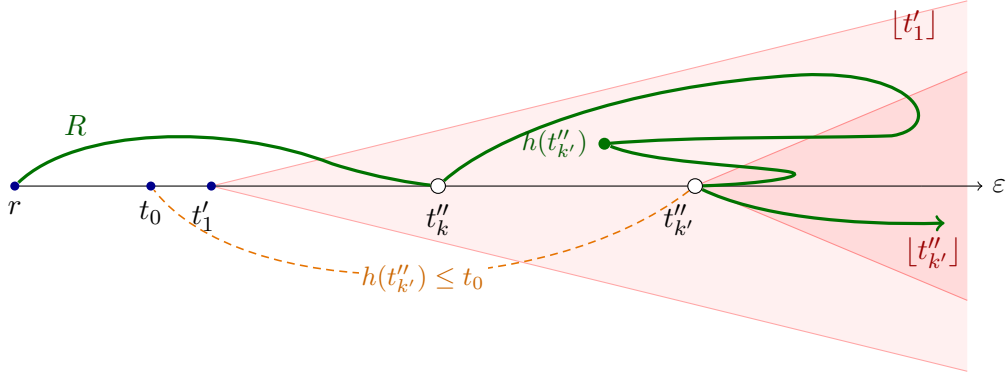

\begin{proof}[Proof of \cref{GammaThm}(i) and~(ii)]
By \cref{GammaTGraph}, $\BG$ is a $T$-graph of finite adhesion, and in particular connected; it has $2^{\aleph_0}$ vertices, as noted after \cref{GammaEnds}.
Statement~(i) is part of \cref{GammaEnds}, and~(ii) follows from \cref{Missed}, as $S$ was an arbitrary spanning tree and $S$ contains a ray of $\eps$ only if it contains a rooted ray of~$\eps$.
\end{proof}

\section{Topological spanning trees}\label{sec_tst}

We recall the space $|G|$ of a graph $G$ from~\cite[Section~2]{DiestelKuhn}.
Its points are the vertices of~$G$, the inner points of its edges, and its ends; every edge $e=xy$ is a copy $[x,y]$ of the unit interval, with interior~$\mathring e$.
Without its ends, $|G|$ is the 1-complex~$G$, in which a vertex $u$ has as basic open neighbourhoods the unions of half-open intervals $[u,z_e)\subseteq e$ with $z_e\in\mathring e$, one for each edge $e$ at~$u$.
The topology \textsc{Top} on $|G|$ is generated by the open sets of this 1-complex together with the following sets $\hat C(X,\eps)$, for all ends $\eps$ and finite sets $X\subseteq V(G)$.
The set $\hat C(X,\eps)$ consists of the vertices of $C(X,\eps)$, the ends whose rays have tails in $C(X,\eps)$, the inner points of edges of $C(X,\eps)$, and a half-open interval $(z,y]$ with $z\in\mathring e$ of every edge $e=xy$ with $x\in X$ and $y\in C(X,\eps)$.
The space $|G|$ is Hausdorff~\cite[Section~2]{DiestelKuhn}.
An \emph{arc} in $|G|$ is a homeomorphic image of $[0,1]$, which we parametrise as an injective continuous map $A\colon[0,1]\to|G|$, and a \emph{circle} is a homeomorphic image of the unit circle.
We use the following facts.
\begin{enumerate}[(T1)]
\item\label{T1} If a path-connected subspace $Y$ of $|G|$ contains no circle, then it contains a unique arc between any two of its points: it contains a path, and hence an arc~\cite[Lemma~2.1]{DiestelKuhn}, and the union of two distinct arcs between the same points would contain a circle, as $|G|$ is Hausdorff; cf.~\cite[p.~839]{DiestelKuhn}.
\item\label{T2} For every edge $e=xy$, the interior $\mathring e$ is open in $|G|$, and its closure is $[x,y]$: a vertex $u\ne x,y$ has a neighbourhood disjoint from~$\mathring e$, and an end $\eps$ has the neighbourhood $\hat C(\{x,y\},\eps)$.
So if an arc $A$ meets $\mathring e$ but $A(0),A(1)\notin\mathring e$, then $A^{-1}(\mathring e)$ is an open interval $(a,b)$ with $\{A(a),A(b)\}=\{x,y\}$; we say that $A$ \emph{traverses}~$e$.
\item\label{T3} Every vertex $u$ has a neighbourhood containing no other vertex and no end, whose points other than $u$ form a disjoint union of open sets each contained in the interior of an edge at~$u$.
So if $A(\tau)=u$ for an arc~$A$, then for small $\delta>0$ the sets $A((\tau,\tau+\delta))$ and $A((\tau-\delta,\tau))$, where defined, each lie in the interior of a single edge at~$u$.
\end{enumerate}

A \emph{region} of a graph $G$ is a connected induced subgraph $C$ such that $N(C)$ is finite.
Then $C$ is a component of $G-N(C)$.
We write $\hat C$ for the set of the vertices of~$C$, the inner points of edges with an endvertex in~$C$, and the ends whose rays have tails in~$C$.
The following lemma is an analogue of the jumping arc lemma~\cite{PitzBook}, which is implicit in~\cite{DiestelKuhn}.

\begin{lemma}[Jumping arcs]\label{JumpingArc}
Let $C$ be a region of~$G$.
Then $\hat C$ is open in~$|G|$, and its closure is contained in $\hat C\cup N(C)$.
In particular, every connected subset of $|G|\setminus N(C)$ that meets $\hat C$ lies in~$\hat C$.
\end{lemma}

\begin{proof}
Let $X:=N(C)$.
A vertex of $C$ has a neighbourhood as in~\ref{T3} that lies in~$\hat C$, as all edges at it have an endvertex in~$C$; an inner point of an edge lies in its interior, which is open by~\ref{T2}; and an end $\eps\in\hat C$ satisfies $C(X,\eps)=C$, so its neighbourhood $\hat C(X,\eps)$ lies in~$\hat C$.
Similarly, every point outside $\hat C\cup X$ has a neighbourhood that avoids~$\hat C$.
For a vertex $u\notin C\cup X$, take a neighbourhood as in~\ref{T3}: no edge at $u$ has an endvertex in~$C$.
For an inner point of an edge without an endvertex in~$C$, take the interior of that edge.
And for an end $\eps\notin\hat C$, the component $C(X,\eps)$ of $G-X$ is not~$C$, so $\hat C(X,\eps)$ avoids~$\hat C$.
Finally, if $Y\subseteq|G|\setminus X$ is connected, then $Y\cap\hat C$ and $Y\setminus\hat C$ are open in~$Y$, by what we have shown; so if $Y$ meets $\hat C$, it lies in~$\hat C$.
\end{proof}

For the rest of this section, $G$ is a $T$-graph of finite adhesion on an order tree $T$ with root~$r$; we apply the results to $\BG$ below and to $\G$ in \cref{sec_refinement}.
The ends of such $T$-graphs and their topology are well understood~\cite[Section~5]{KurkofkaPitz}, \cite[Section~2]{PitzRaySpaces}, and the lemmas below adapt this to the space~$|G|$; for completeness, we state and prove what we need.
The \emph{children} of a node $v$ are the nodes $s$ with $s^\star=v$.
Every node above $v$ lies in the up-closure of a unique child of~$v$, and no edge joins the up-closures of two distinct children, as their nodes are incomparable.
For every successor~$s$, the up-closure $\upcl s$ induces a region of~$G$, by~\ref{K1}, with $N(\upcl s)\subseteq\sdowncl s$ by~\ref{K3}; we write $\hat C_s$ for $\hat C$ with $C:=G[\upcl s]$.
So the vertices in $\hat C_s$ are those in~$\upcl s$, and by~\eqref{TailCone}, an end $\eps$ lies in $\hat C_s$ if and only if $s\in\varrho(\eps)$.

\begin{lemma}\label{ArcExit}
Let $v\ne r$ be a vertex, and let $A$ be an arc from $v$ to~$r$.
Then $A$ has a first vertex not in~$\upcl v$, and this vertex lies in~$\sdowncl v$.
\end{lemma}

\begin{proof}
By~\ref{T3} and~\ref{T2}, the arc $A$ begins by traversing an edge $vy$.
If $y<v$, then $y$ is the desired vertex.
Otherwise $y\in\upcl s$ for a child $s$ of~$v$.
The subarc of $A$ from $y$ to $r$ meets $\hat C_s$ but does not lie in it, as $r\notin\upcl s$; so it meets $N(\upcl s)$, by \cref{JumpingArc}, and we let $w$ be its first vertex in $N(\upcl s)$.
Before~$w$, this subarc lies in~$\hat C_s$, by \cref{JumpingArc}; so all vertices of $A$ before $w$ lie in $\{v\}\cup\upcl s\subseteq\upcl v$; see \cref{fig_endbelowlimit}, with $v=\hat u$.
And $w\ne v$, as $A$ is injective, so $w\in N(\upcl s)\setminus\{v\}\subseteq\sdowncl v$.
\end{proof}

\begin{lemma}\label{LastEntry}
Let $s$ be a successor, $\eps$ an end with $s\in\varrho(\eps)$, and $A$ an arc from a vertex not in $\upcl s$ to~$\eps$.
Then $A$ traverses an edge $xy$ from a vertex $x\in N(\upcl s)$ to a vertex $y\in\upcl s$, and all vertices of $A$ after $x$ lie in~$\upcl s$.
\end{lemma}

\begin{proof}
The arc $A$ ends in $\hat C_s$ but does not lie in it.
So it meets the finite set $N(\upcl s)$, by \cref{JumpingArc}, and we let $x$ be its last vertex in this set, which exists as $A$ is injective.
After~$x$, the arc lies in~$\hat C_s$, by \cref{JumpingArc}; so its vertices after $x$ lie in~$\upcl s$.
By~\ref{T3} and~\ref{T2}, the arc $A$ traverses an edge $xy$ right after~$x$, whose interior lies in~$\hat C_s$; so $y\in\upcl s$.
\end{proof}

\begin{proof}[Proof of \cref{GammaThm}(iii)]
Let $Z\subseteq|\BG|$ be a path-connected subspace that contains all vertices and ends of~$\BG$, and suppose that $Z$ contains no circle.
By~\ref{T1}, for every vertex $v\ne r$ the space $Z$ contains a unique arc from $v$ to~$r$, which we denote by $vZr$, as for the paths $vSr$ in \cref{sec_proof}.
Let $h(v)$ be the first vertex of $vZr$ not in~$\upcl v$; then $h(v)<v$ by \cref{ArcExit}.

We now go through \cref{sec_proof} with this map $h$ in place of the exits of the tree~$S$.
The proof of \cref{FirstDiagonal} uses only that $h(v)$ is defined for every limit~$v$ of~$\Tc$ and that $h(v)<v$.
The second diagonal involves $h$ only through the choice of $t_0$ in \cref{FirstDiagonal}, and \cref{Bookkeeping,Bottleneck} do not involve $S$ or $h$ at all.
So we obtain a node $t_0$ as in \cref{FirstDiagonal} for our map~$h$, and with it the nodes $t''_n$ and $t'_n$, the high-ray $\hat\varrho$ and its end $\eps$ as in \cref{sec_proof}.

As $Z$ is path-connected, it contains an arc $A$ from $r$ to~$\eps$.
We follow the proof of \cref{Missed}.
Where that proof uses \cref{Ends}(iii), we show instead that the vertices of a final segment of $A$ cannot all lie above a top of~$\hat\varrho$; this argument goes back to our unpublished notes from 2023, where it appears in a sketch showing that no regressive spanning tree of~$\RG$ induces a topological spanning tree.
Let $x_1y_1$ be the edge given by \cref{LastEntry} for the successor $t'_1\in\hat\varrho=\varrho(\eps)$, and let $\tau_1$ be the parameter of~$y_1$; so all vertices of $A$ from $\tau_1$ on lie in~$\upcl{t'_1}$.

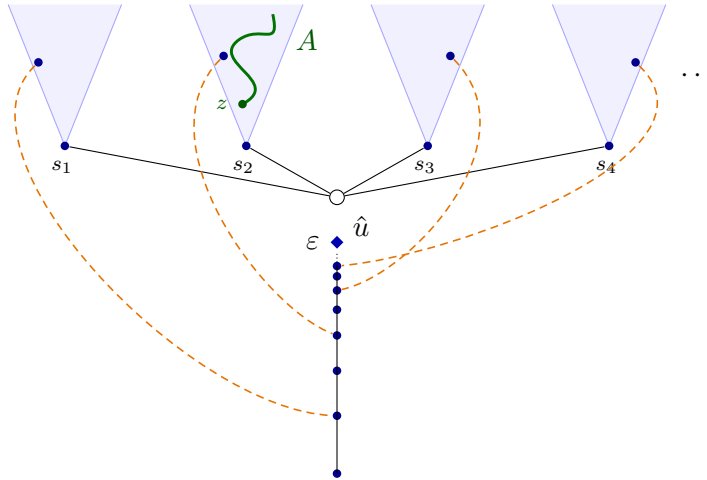
\begin{figure}[h!]
\centering
\begin{tikzpicture}[
    dot/.style={circle,fill=blue!55!black,inner sep=0pt,minimum size=3.2pt},
    lim/.style={circle,draw=black,fill=white,thin,inner sep=0pt,minimum size=5.5pt},
    endp/.style={diamond,fill=blue!70!black,inner sep=0pt,minimum size=5pt},
    downE/.style={orange!90!black,semithick,densely dashed},
    yscale=0.85
  ]
\foreach \i/\y in {1/0,2/0.9,3/1.6,4/2.15,5/2.55,6/2.85,7/3.07,8/3.23}
  \node[dot] (c\i) at (0,\y) {};
\draw (0,0) -- (0,3.23);
\draw[densely dotted] (0,3.23) -- (0,3.45);
\node[endp,label=left:{$\eps$}] (e) at (0,3.6) {};
\node[lim,label={[yshift=-1pt]below right:{$\hat u$}}] (v) at (0,4.3) {};
\foreach \i/\x in {1/-3.6,2/-1.2,3/1.2,4/3.6}{
  \fill[blue!6] (\x,5.1) -- (\x-0.75,7.3) -- (\x+0.75,7.3) -- cycle;
  \draw[blue!35] (\x-0.75,7.3) -- (\x,5.1) -- (\x+0.75,7.3);
  \node[dot,label={[xshift=-1pt]below:{\scriptsize $s_\i$}}] (s\i) at (\x,5.1) {};
  \draw (v) -- (s\i);
}
\node at (4.8,6.2) {$\cdots$};
\node[dot] (a1) at (-3.95,6.4) {};
\node[dot] (a2) at (-1.5,6.5) {};
\node[dot] (a3) at (1.5,6.5) {};
\node[dot] (a4) at (3.95,6.4) {};
\draw[downE] (a1) .. controls (-5.2,5.0) and (-2.0,1.0) .. (c2);
\draw[downE] (a2) .. controls (-2.6,5.2) and (-1.0,2.6) .. (c4);
\draw[downE] (a3) .. controls (2.6,5.2) and (1.0,3.1) .. (c6);
\draw[downE] (a4) .. controls (5.2,5.0) and (1.6,3.4) .. (c8);
\draw[green!45!black,line width=1.1pt]
  (-1.25,5.75) .. controls (-0.7,6.1) and (-1.7,6.4) .. (-1.3,6.75) .. controls (-1.0,7.0) and (-0.75,6.6) .. (-0.85,7.15);
\node[dot,fill=green!35!black,label={[green!35!black]left:{\scriptsize $z$}}] at (-1.25,5.75) {};
\node[green!35!black] at (-0.4,6.75) {$A$};
\end{tikzpicture}
\caption{An arc can leave the up-closure of a child~$s$ of a limit~$\hat u$ only through $N(\upcl s)\subseteq\downcl{\hat u}$.}
\label{fig_endbelowlimit}
\end{figure}

\emph{For every $\tau\ge\tau_1$ such that $A(\tau)$ is a vertex, $A$ meets one of the limits~$t''_m$ at or after~$\tau$.}
First, some vertex of $A$ at or after $\tau$ lies outside $\bigcap_{n\ge1}\upcl{t'_n}$.
Otherwise this intersection contains the vertex $A(\tau)$, which lies above all of~$\hat\varrho$; so the least element of $\downcl{A(\tau)}\setminus\hat\varrho$ is a top of~$\hat\varrho$, which can only be the sequence $\hat u:=\bigcup_n t'_n$, and the intersection is~$\upcl{\hat u}$.
If $A$ visits $\hat u$ at or after~$\tau$, let $z$ be the vertex at the other end of the edge that $A$ traverses right after~$\hat u$, by~\ref{T3} and~\ref{T2}; otherwise let $z:=A(\tau)$.
Then the subarc $zA\eps$ of $A$ avoids~$\hat u$, and all its vertices lie in $\upcl{\hat u}$; in particular $z\in\upcl s$ for a child $s$ of~$\hat u$.
But $s\notin\hat\varrho=\varrho(\eps)$, so $\eps\notin\hat C_s$, and \cref{JumpingArc} gives a vertex of $zA\eps$ in $N(\upcl s)\subseteq\downcl{\hat u}$, a contradiction; see \cref{fig_endbelowlimit}.
Now let $x$ be a vertex of $A$ at or after $\tau$ outside $\bigcap_{n\ge1}\upcl{t'_n}$, with $x\notin\upcl{t'_{n+1}}$, say.
As $t'_{n+1}\in\varrho(\eps)$, we have $\eps\in\hat C_{t'_{n+1}}$, so \cref{JumpingArc} gives a vertex $u$ of $xA\eps$ in $N(\upcl{t'_{n+1}})$.
Then $u\in\{t''_0,\dots,t''_n\}\cup A^\star_{t_0}$ by \cref{Bottleneck}, see \cref{fig_diagonal}, and as $u$ is a vertex of $A$ at or after $\tau_1$, it lies in $\upcl{t'_1}$, which is disjoint from~$A^\star_{t_0}$; so $u=t''_m$ for some~$m\le n$.

Applying this repeatedly, starting with $\tau=\tau_1$ and continuing each time at a vertex of $A$ after the limit just found, which exists by~\ref{T3} and~\ref{T2}, we obtain limits $t''_{m_1},t''_{m_2},\dots$ that $A$ meets at strictly increasing parameters at or after~$\tau_1$; in particular, they are distinct.
The indices $m_1,m_2,\dots$ cannot decrease forever; so $A$ meets some $t''_k$ before some $t''_{k'}$ with $k<k'$, both at or after~$\tau_1$; see \cref{fig_twopassages}, with $A$ in place of~$R$.
By~\ref{T1}, the arc $t''_{k'}Zr$ is the initial segment of $A$ up to $t''_{k'}$, traversed backwards.
It contains $t''_k\notin\upcl{t''_{k'}}$, so its first vertex outside $\upcl{t''_{k'}}$, which is $h(t''_{k'})$, comes no later than~$t''_k$; and $h(t''_{k'})\ne t''_k$, as $h(t''_{k'})\le t_0<t''_k$.
So $A$ meets $h(t''_{k'})$ after $t''_k$, and hence after~$\tau_1$.
But all vertices of $A$ after $\tau_1$ lie in $\upcl{t'_1}$, while $h(t''_{k'})\le t_0<t'_1$, a contradiction.
%This completes the proof of \cref{GammaThm}(iii).
\end{proof}

\section{Thin ends}\label{sec_refinement}

We now give the formal definition of $\G$ from the introduction.

\subsection*{Sequences with bits}
Let $S^*$ be the set of all sequences $t\colon\alpha\to\N\times\{0,1\}$ with $\alpha<\omega_1$ whose first coordinates, the \emph{values} of~$t$, are pairwise distinct; the second coordinates are the \emph{bits} of~$t$.
Ordered by extension, $S^*$ is an order tree, and forgetting the bits maps it onto~$T^*$.
We transfer the notation of \cref{sec_Gamma} via the values: $\im(t)$ is the set of values of~$t$, and $\last(s)$, $s^\star$, $\Sigma(S^*)$, co-infinite nodes and the sets $A_t$ and $A^\star_t$ of records are defined as there; \cref{Records} and~$(\star)$ hold for $S^*$ with the same proofs.
For $s\in\Sigma(S^*)$ we write $b(s)$ for the bit of its last entry, and $t\conc(n,b)$ is $t$ followed by the entry $(n,b)$.
Every high-ray $\varrho$ of $S^*$ has at most one top, the sequence~$\bigcup\varrho$.
A record $x$ of a node $t$ is a \emph{$0$-record} if $b(x)=0$, and a limit $\mu$ other than the root is \emph{open} if it has infinitely many $0$-records, and \emph{closed} otherwise.
By \cref{Records}(iii), the $0$-records of an open limit $\mu$ are cofinal in~$\sdowncl\mu$.

\subsection*{The tree \texorpdfstring{$\Tb$}{T-flat}}
Let $\Tb$ consist of all co-infinite $x\in S^*$ all of whose limit initial segments other than the root are open; this includes $x$ itself if $x$ is a limit.
Then $\Tb$ is a down-closed subtree of $S^*$, and every high-ray of $\Tb$ has at most one top; so $\Tb$ is Hausdorff.
We put $\Sigma(\Tb):=\Sigma(S^*)\cap\Tb$.
The tree $\Tb$ contains every co-infinite sequence all of whose bits are $0$, as every record is then a $0$-record; so $\Tb$ contains a copy of the tree $\Tc$ of \cref{sec_Gamma}, and in particular it has height~$\omega_1$.

\subsection*{The graph \texorpdfstring{$\G$}{G-flat}}
The graph $\G$ has vertex set $\Tb$ and the following edges:
\begin{enumerate}[(E1)]
\item\label{E1} $q^\star q$ for every $q\in\Sigma(\Tb)$;
\item\label{E2} $tx^\star$ for every node $t\in\Tb$ other than the root and every $0$-record $x$ of~$t$.
\end{enumerate}
So $\G$ is defined like $\BG$, except that only the records with bit~$0$ count, and the edges~\ref{E1} also join the successors with bit~$1$ to their predecessors.
A successor $s$ with $b(s)=1$ is a record of many nodes above it, but none of them is joined to $s^\star$ on account of~$s$; this invisibility is what we shall exploit.
It costs us some limits: a limit with only finitely many $0$-records has too few down-neighbours, and we have removed it, together with everything above it, by passing to~$\Tb$.

\begin{lemma}\label{TGraph}
The graph $\G$ is a $T$-graph of finite adhesion on~$\Tb$.
More precisely, for every $q\in\Sigma(\Tb)$, every vertex in $N(\upcl q)$ other than $q^\star$ is the immediate predecessor of a $0$-record of~$q$, and so $|N(\upcl{q})|\le\last(q)+1$.
\end{lemma}

\begin{proof}
As in the proof of \cref{GammaTGraph}.
The ends of every edge are comparable.
Every $q\in\Sigma(\Tb)$ is adjacent to $q^\star$, and every limit $\mu\ne\emptyset$ of $\Tb$ is open, so its $0$-records, and hence their immediate predecessors, are cofinal in $\sdowncl{\mu}$ by \cref{Records}(iii).
So $\G$ is a $T$-graph.
Now let $q\in\Sigma(\Tb)$, and let $u\in\upcl q$ be adjacent to $x\notin\upcl q$; then $x<q$ by~\ref{K3}, and $x<u$.
If the edge $xu$ has type~\ref{E1}, then $x=u^\star$, which forces $u=q$ and $x=q^\star$.
Otherwise $x=y^\star$ for a $0$-record $y$ of~$u$; as $y^\star<q$, we have $y\le q$, so $y$ is a $0$-record of~$q$, by~$(\star)$ if $y<q$.
By \cref{Records}(ii), $q$ has at most $\last(q)$ records.
\end{proof}

As $\Tb$ is Hausdorff, every end of $\G$ has countable combined degree, by the remark after \cref{Dominators}.

\subsection*{Proof of \texorpdfstring{\cref{MainThm}}{Theorem 1.7}}
Throughout this subsection, $S$ is a spanning tree of~$\G$, which we regard as rooted at the root $r$ of~$\Tb$; we define the exits $h(v)$ of the vertices $v\ne r$ as in \cref{sec_proof}, and \cref{Exit} holds with the same proof.
We find an end of $\G$ of degree~$1$ that is dominated by at least one but only finitely many vertices, and missed by~$S$.
We now run \cref{sec_proof} in~$\Tb$, with the following changes.

\emph{Extensions.}
For $t\in\Sigma(\Tb)$, an \emph{extension} of~$t$ is a node $t'\ge t$ of~$\Tb$ all of whose values after~$t$ exceed~$\last(t)$, and it is a \emph{$1$-extension} of~$t$ \emph{skipping}~$a_1$ if moreover $t'>t$ and all these values exceed~$a_1$, the least value above $\last(t)$ that does not occur in~$t$.
The bits play no role here.

\emph{The first diagonal.}
The claim now reads: every $t\in\Sigma(\Tb)$ has an extension $t'\in\Sigma(\Tb)$ such that every $1$-extension $t''$ of $t'$ that is a limit of~$\Tb$ satisfies $h(t'')>t$.
As in \cref{FirstDiagonal}, the claim is false: there exists $t_0\in\Sigma(\Tb)$ such that every extension $t'\in\Sigma(\Tb)$ of $t_0$ has a $1$-extension $t''$ that is a limit of~$\Tb$ with $h(t'')\le t_0$.
The proof of \cref{FirstDiagonal} goes through, starting with an arbitrary $s_0\in\Sigma(\Tb)$, with $s_{n+1}:=s'_n\conc(a_{2n+2},0)$, which lies in $\Tb$ as its limit initial segments are those of~$s'_n$.
The only new point is that the limit $\hat s=\bigcup_n s_n$ lies in~$\Tb$; it is co-infinite as in the proof of \cref{FirstDiagonal}.
By the proof, the values of $\hat s$ after $s'_n$ are at least $a_{2n+2}=\last(s_{n+1})$, so every $s_{n+1}$ is a $0$-record of $\hat s$, and $\hat s$ is open.
Its other limit initial segments are initial segments of some $s_n\in\Tb$, so they are open too, and $\hat s$ is a limit of~$\Tb$.

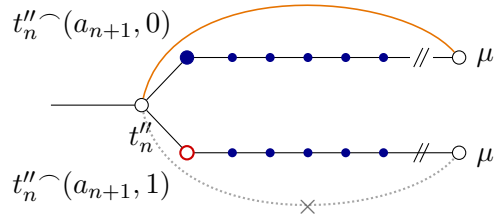
\begin{figure}[h!]
\centering
\begin{tikzpicture}[
    dot/.style={circle,fill=blue!55!black,inner sep=0pt,minimum size=3.4pt},
    zero/.style={circle,draw=blue!55!black,fill=blue!55!black,inner sep=0pt,minimum size=5pt},
    bnd/.style={circle,draw=red!80!black,fill=white,thick,inner sep=0pt,minimum size=5pt},
    lim/.style={circle,draw=black,fill=white,thin,inner sep=0pt,minimum size=5pt},
    edgeE/.style={orange!90!black,semithick},
    noE/.style={gray!70,densely dotted,thick},
    yscale=0.9
  ]
\draw (-1.2,0) -- (0,0);
\node[lim,label=below:{$t''_n$}] (x) at (0,0) {};
% bit 0 branch (up)
\draw (x) -- (0.6,0.7) -- (4.2,0.7);
\node[zero,label={[yshift=1pt]above left:{$t''_n\conc(a_{n+1},0)$}}] (z) at (0.6,0.7) {};
\foreach \u in {1.2,1.7,2.2,2.7,3.2} \node[dot] at (\u,0.7) {};
\foreach \yy in {0.7,-0.7}{\fill[white] (3.55,\yy-0.1) rectangle (3.85,\yy+0.1);
  \draw ($(3.7,\yy)+(-0.10,-0.14)$) -- ($(3.7,\yy)+(0.02,0.14)$);
  \draw ($(3.7,\yy)+(-0.02,-0.14)$) -- ($(3.7,\yy)+(0.10,0.14)$);}
\node[lim,label=right:{$\mu$}] (m0) at (4.2,0.7) {};
\draw[edgeE] (x) .. controls (0.4,1.9) and (3.4,1.7) .. (m0);
% bit 1 branch (down)
\draw (x) -- (0.6,-0.7) -- (4.2,-0.7);
\node[bnd,label={[yshift=-1pt]below left:{$t''_n\conc(a_{n+1},1)$}}] (o) at (0.6,-0.7) {};
\foreach \u in {1.2,1.7,2.2,2.7,3.2} \node[dot] at (\u,-0.7) {};
\node[lim,label=right:{$\mu'$}] (m1) at (4.2,-0.7) {};
\draw[noE] (x) .. controls (0.4,-1.9) and (3.4,-1.7) .. (m1);
\node[gray!80!black] at (2.2,-1.5) {$\times$};
\end{tikzpicture}
\caption{The bit of a new node decides whether the later limits are joined to its predecessor.}
\label{fig_bitchoice}
\end{figure}

\emph{The second diagonal.}
Fix $t_0$ as above, and let $t'_0:=t_0\conc(a_0,0)$, where $a_0$ is the least value above $\last(t_0)$ that does not occur in~$t_0$; the node $t'_0$ is an extension of $t_0$ in $\Sigma(\Tb)$.
Recursively, for $n\ge0$, let $t''_n$ be a $1$-extension of $t'_n$ skipping~$a_{n+1}$ that is a limit of~$\Tb$ with $h(t''_n)\le t_0$, and append the skipped value with bit~$1$ (see \cref{fig_bitchoice}):
\[
t'_{n+1}:=t''_n\conc(a_{n+1},1).
\]
The boundaries $t'_1<t'_2<\cdots$ and the stages are defined as in \cref{sec_proof}, and the analogue of \cref{Bookkeeping} holds: the recursion is well defined, and for every $n\ge0$, (i)~every node $y$ in the $n$th stage satisfies $\last(y)>a_{n+1}$; (ii)~$\last(t'_0)<a_{n+1}<a_{n+2}$; (iii)~$t'_{n+1}\in\Sigma(\Tb)$ is an extension of~$t_0$ with $\last(t'_{n+1})=a_{n+1}$.
The proof is that of \cref{Bookkeeping}, with the induction hypothesis $\last(t'_n)\ge\last(t'_0)$; in~(iii), the node $t'_{n+1}$ lies in $\Tb$, as its limit initial segments are those of~$t''_n$, and it is an extension of $t_0$, as its values after $t_0$ are $\last(t'_0)=a_0>\last(t_0)$ and values exceeding~$\last(t'_0)$.
Let $\hat\varrho:=\bigcup_{n\ge1}\downcl{t'_n}$, a high-ray of $\Tb$ containing $t'_0$, all $t''_n$ and all boundaries, and let $\eps:=\eps_{\hat\varrho}$ be its end.

The following finite set of nodes at or below $t'_0$ will determine everything through which $S$ could reach $\eps$ from below:
let $M$ be the set of $0$-records of~$t'_0$.
Note that $t'_0\in M$, and that $M$ has at most $\last(t'_0)$ elements by \cref{Records}(ii).
We write $M^\star:=\{x^\star:x\in M\}$; as all elements of $M$ are at most $t'_0$, and $(t'_0)^\star=t_0$, all elements of $M^\star$ are at most~$t_0$.

The decisive property of the boundaries is that they kill all records in the stages before them, while being invisible themselves.

\begin{lemma}
\label{RecordKilling}
Let $\mu\in\Tb$, and let $x$ be a $0$-record of $\mu$ with $x<t'_{n+1}\le\mu$ for some $n\ge0$.
Then $x\in M$.
\end{lemma}

\begin{proof}
Suppose first that $x>t'_0$.
As $x$ is a node of $\Sigma(\Tb)$ below $t'_{n+1}$, either $x$ is a boundary $t'_{k+1}$ with $k<n$, or $x$ lies in the $k$th stage for some $k\le n$.
A boundary has bit~$1$ and is not a $0$-record.
If $x$ lies in the $k$th stage, then $\last(x)>a_{k+1}=\last(t'_{k+1})$ by the analogue of \cref{Bookkeeping}(i), and $x<t'_{k+1}\le\mu$, so $x$ is not a record of~$\mu$.
Hence $x\le t'_0$, and so $x$ is a $0$-record of~$t'_0$, by~$(\star)$ if $x<t'_0$.
So $x\in M$.
\end{proof}

\begin{figure}[h!]
\centering
\begin{tikzpicture}[
    dot/.style={circle,fill=blue!55!black,inner sep=0pt,minimum size=3.4pt},
    bnd/.style={circle,draw=red!80!black,fill=white,thick,inner sep=0pt,minimum size=5pt},
    lim/.style={circle,draw=black,fill=white,thin,inner sep=0pt,minimum size=5pt},
    sq/.style={rectangle,fill=blue!55!black,inner sep=0pt,minimum size=3.6pt},
    edgeE/.style={orange!90!black,semithick},
    hasseE/.style={orange!90!black,line width=1.6pt},
    exitE/.style={orange!90!black,semithick,densely dashed},
    noE/.style={gray!70,densely dotted,thick},
  ]
% clip to the drawing: the control points of the curves would enlarge the bounding box
\clip (-0.25,-3.1) rectangle (13.5,2.5);
\fill[red!7] (7.85,0) -- (13.2,2.0) -- (13.2,-2.0) -- cycle;
\draw[red!35] (13.2,2.0) -- (7.85,0) -- (13.2,-2.0);
\node[red!60!black] at (12.7,1.2) {$\upcl{t'_2}$};
\draw (0,0) -- (3.55,0);
\draw (3.85,0) -- (6.85,0);
\draw (7.15,0) -- (10.15,0);
\draw[->] (10.45,0) -- (13.0,0) node[right] {$\eps$};
\foreach \bx in {3.7,7.0,10.3}{
  \draw ($(\bx,0)+(-0.10,-0.14)$) -- ($(\bx,0)+(0.02,0.14)$);
  \draw ($(\bx,0)+(-0.02,-0.14)$) -- ($(\bx,0)+(0.10,0.14)$);
}
\node[dot,label=below:{$r$}] at (0,0) {};
\node[sq] (M1) at (0.55,0) {};
\node[sq] (M2) at (1.0,0) {};
\node[sq,label={[yshift=-1pt]below:{$t_0$}}] (t0) at (1.45,0) {};
\node[dot,label={[yshift=1pt,xshift=3pt]above:{$t'_0$}}] (tp) at (1.85,0) {};
\foreach \x in {2.3,2.75,3.2, 4.95,5.4,5.85,6.25,6.6, 8.3,8.7,9.1,9.5,9.9, 11.6,12.0,12.4}
  \node[dot] at (\x,0) {};
\node[lim,label={[yshift=1pt]above:{$t''_0$}}] (l0) at (3.95,0) {};
\node[bnd,label={[yshift=-1pt]below:{$t'_1$}}] (b1) at (4.55,0) {};
\node[lim,label={[yshift=3pt,xshift=-4pt]above:{$t''_1$}}] (l1) at (7.25,0) {};
\node[bnd,label={[yshift=-1pt]below:{$t'_2$}}] (b2) at (7.85,0) {};
\node[lim,label={[yshift=1pt]above:{$t''_2$}}] (l2) at (10.55,0) {};
\node[bnd,label={[yshift=-1pt]below:{$t'_3$}}] (b3) at (11.15,0) {};
\node[gray] at (2.75,-0.42) {\scriptsize stage $0$};
\node[gray] at (5.75,-0.42) {\scriptsize stage $1$};
\node[gray] at (9.05,-0.42) {\scriptsize stage $2$};
% ways into the cone above t'_2: the edge t''_1 t'_2, and 0-records in M
\draw[hasseE] (l1) -- (b2);
\draw[edgeE] (t0) .. controls (4.0,2.5) and (8.8,2.5) .. (l2);
\draw[edgeE] (M2) .. controls (3.4,3.25) and (9.0,3.1) .. (l2);
% invisible boundary: bit 1, so no edge from later limits
\draw[noE] (l0) .. controls (6.4,1.6) and (9.0,1.6) .. (l2);
\node[gray!80!black] at (7.85,1.22) {$\times$};
% t'_0 is a 0-record of every t''_n
\draw[edgeE] (t0) .. controls (2.4,1.0) and (3.4,1.0) .. (l0);
\draw[edgeE] (t0) .. controls (3.2,1.7) and (6.2,1.7) .. (l1);
% exits
\draw[exitE] (l0) .. controls (3.2,-1.0) and (2.0,-1.0) .. (t0);
\draw[exitE] (l1) .. controls (5.5,-2.0) and (1.6,-2.0) .. (M2);
\draw[exitE] (l2) .. controls (8.0,-3.0) and (1.4,-3.0) .. (M1);
\node[orange!80!black] at (5.6,-1.85) {\scriptsize exits $h(t''_n)\le t_0$};
\end{tikzpicture}
\caption{Bottlenecks in~$\G$: the up-closure of $t'_{n+1}$ can be entered only from~$t''_n$ or from the finite set $M^\star$ (squares), and $t_0$ dominates~$\eps$.}
\label{fig_bits}
\end{figure}
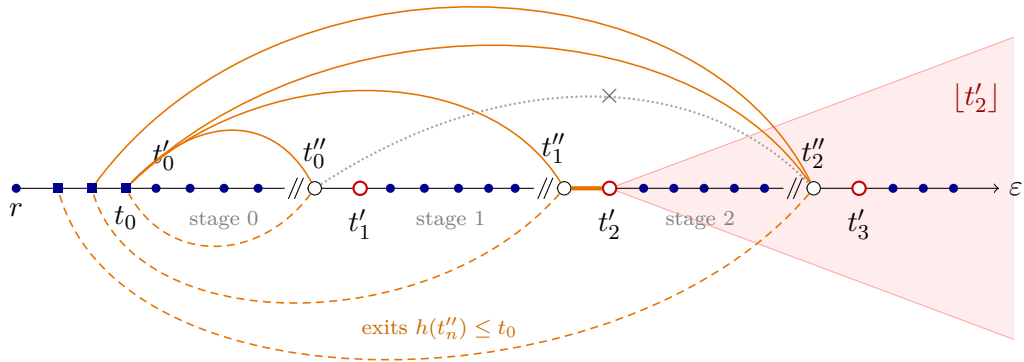

\begin{lemma}[Bottlenecks in $\G$]\label{BottleneckFlat}
For every $n\ge0$, we have $N(\upcl{t'_{n+1}})\subseteq\{t''_n\}\cup M^\star$.
\end{lemma}

\begin{proof}
By \cref{TGraph}, a vertex $x\ne t''_n$ of $N(\upcl{t'_{n+1}})$ is the predecessor of a $0$-record $y$ of~$t'_{n+1}$.
As $t'_{n+1}$ has bit~$1$, we have $y<t'_{n+1}$, so $y\in M$ by \cref{RecordKilling}, and $x\in M^\star$.
\end{proof}

So the only ways into the up-closure of $t'_{n+1}$ lead through $t''_n$ or through the finitely many vertices in~$M^\star$, which lie at or below~$t_0$.
In~$\BG$, each boundary $t'_{k+1}$ with $k<n$ would be a record of the later limits, and its predecessor $t''_k$ would offer a further way in; see \cref{fig_bits}.

\begin{lemma}\label{NoTopFlat}
The high-ray $\hat\varrho$ has no top, and $\bigcap_{n\ge1}\upcl{t'_n}=\emptyset$.
\end{lemma}

\begin{proof}
A top of $\hat\varrho$ would be a limit of $\Tb$, and so it would be the sequence $\hat u:=\bigcup_n t'_n$.
As in the proof of \cref{RecordKilling}, every record of $\hat u$ above $t'_0$ is a boundary, and so it has bit~$1$; and only finitely many records of $\hat u$ lie at or below $t'_0$, by \cref{Records}(iii).
So $\hat u$ is closed and does not lie in~$\Tb$.
If a vertex $v$ lay in every $\upcl{t'_n}$, then all of $\hat\varrho$ would lie below $v$ while $v\notin\hat\varrho$, so the least element of $\downcl{v}\setminus\hat\varrho$ would be a top of~$\hat\varrho$.
\end{proof}

\begin{lemma}\label{MissedFlat}
No rooted ray of $S$ belongs to~$\eps$.
\end{lemma}

\begin{proof}
The proof of \cref{Missed} applies, with \cref{BottleneckFlat} in place of \cref{Bottleneck}: the vertex $u$ lies in $\{t''_n\}\cup M^\star$, and $M^\star$ is disjoint from $\upcl{t'_1}$, as $M^\star\subseteq\downcl{t_0}$.
\end{proof}

\begin{lemma}\label{MissedProfile}
The end $\eps$ has degree~$1$ and is dominated by at least one, but at most $|M|$, vertices.
\end{lemma}

\begin{proof}
\emph{The end $\eps$ is dominated by $t_0$.}
Every node of $\hat\varrho$ above $t'_0$, in particular every $t''_n$, has $t'_0$ as a $0$-record: $t'_0$ has bit $0$, and all values after $t'_0$ exceed $\last(t'_0)$ by the analogue of \cref{Bookkeeping}.
So every $t''_n$ is adjacent to $(t'_0)^\star=t_0$.
By~\ref{K1}, every interval $[t'_{n+1},t''_{n+1}]$ contains a $t'_{n+1}$--$t''_{n+1}$ path, and together with the edges $t''_nt'_{n+1}$ these paths form a ray starting at $t''_0$ whose vertex set is contained and cofinal in~$\hat\varrho$.
This ray belongs to $\eps$ by \cref{Ends}(iii), and $t_0$ is adjacent to all the $t''_n$ on it.

\emph{Every vertex dominating $\eps$ lies in~$M^\star$.}
By \cref{NoTopFlat,Dominators}, such a vertex lies in $N(\upcl{t'_{n+1}})$ for all large~$n$, and hence in $\{t''_n\}\cup M^\star$ by \cref{BottleneckFlat}; as the $t''_n$ are distinct, it lies in~$M^\star$.

\emph{The end $\eps$ has degree~$1$.}
Let $R_1,\dots,R_k$ be disjoint rays of~$\eps$; as $M^\star$ is finite, we may assume that they avoid~$M^\star$, by passing to tails.
By \cref{NoTopFlat}, no vertex lies in all up-closures $\upcl{t'_n}$; so there is an $n$ such that none of the first vertices of $R_1,\dots,R_k$ lies in $\upcl{t'_{n+1}}$.
Each $R_i$ has a tail in $\upcl{t'_{n+1}}$ by~\eqref{TailCone}, so it meets $N(\upcl{t'_{n+1}})\subseteq\{t''_n\}\cup M^\star$, and hence contains~$t''_n$.
As the $R_i$ are disjoint, $k\le1$.
\end{proof}

\begin{proof}[Proof of \cref{MainThm}]
The graph $\G$ is connected, as it is a $T$-graph by \cref{TGraph}, and~(i) was shown after \cref{TGraph}.
For~(ii), let $S$ be a spanning tree of~$\G$.
The end $\eps$ constructed above is missed by $S$, by \cref{MissedFlat}, as $S$ contains a ray of $\eps$ only if it contains a rooted ray of~$\eps$; and it has degree~$1$ and is dominated by at least one but only finitely many vertices, by \cref{MissedProfile}.
For~(iii), the proof of \cref{GammaThm}(iii) in \cref{sec_tst} applies to~$\G$, a $T$-graph of finite adhesion by \cref{TGraph}, with the changes above.
Again we use \cref{BottleneckFlat} and the fact that $M^\star$ is disjoint from $\upcl{t'_1}$; and the vertex $A(\tau)$ cannot lie in every up-closure $\upcl{t'_n}$, by \cref{NoTopFlat}.
\end{proof}

\section{An open problem}\label{sec_open}

Our arguments show that every spanning tree of $\BG$ or of $\G$ misses an end; they never force an end to contain two rooted rays of a spanning tree.
This suggests a weaker notion: a rooted spanning tree $S$ of $G$ is \emph{end-injective} if every end of $G$ contains at most one rooted ray of~$S$.
The notion is due to Halin, who called such subgraphs \emph{endentreu}~\cite[p.~128]{Halin64}, and it was studied in detail by Polat, who called them \emph{terminalement fid\`eles}~\cite[p.~286]{Polat82} and later \emph{end-respecting}~\cite[\S2.2]{Polat96}.

\begin{problem}\label{ProbEndInjective}
Does every connected graph in which every end has countable degree or is dominated by only countably many vertices have an end-injective spanning tree?
\end{problem}

Some restriction on the ends is needed: in a one-ended graph, an end-injective spanning tree is rayless or end-faithful, and the graphs of Seymour and Thomas~\cite{SeymourThomas} and of Thomassen~\cite{Thomassen} have neither.

Polat answered \cref{ProbEndInjective} positively for one-ended graphs~\cite[Theorem~2.11]{Polat90}, and more generally for terminally accessible graphs~\cite[Corollary~2.3]{Polat90}, which include all graphs without uncountably many disjoint rays~\cite[Corollary~2.1]{Polat90}; in fact, he showed that these graphs even have end-faithful spanning trees, see also~\cite[p.~320]{Diestel92}.

By \cref{GammaThm,MainThm}, our graphs $\BG$ and $\G$ are therefore terminally inaccessible.
Nevertheless, both have end-injective regressive spanning trees, as can be seen by joining every successor to its immediate predecessor, and every limit to the immediate predecessor of its least record in $\BG$, or of its least $0$-record in~$\G$.

\end{document}